\documentclass[10pt,a4paper]{amsart}
\usepackage{mathrsfs}
\usepackage{syntonly}
\usepackage{amsmath}
\usepackage{amsthm}
\usepackage{amsfonts}
\usepackage{amssymb}
\usepackage{latexsym}
\usepackage{amscd,amssymb,amsopn,amsmath,amsthm,graphics,amsfonts,mathrsfs,accents,enumerate,verbatim,calc}
\usepackage[dvips]{graphicx}
\usepackage[colorlinks=true,linkcolor=red,citecolor=blue]{hyperref}
\usepackage[all]{xy}
\usepackage{enumitem}
\usepackage{mathtools}
\usepackage{tikz}
\usetikzlibrary{decorations.pathreplacing,decorations.markings}
 \tikzset{
  on each segment/.style={
    decorate,
    decoration={
      show path construction,
      moveto code={},
      lineto code={
        \path [#1]
        (\tikzinputsegmentfirst) -- (\tikzinputsegmentlast);
      },
      curveto code={
        \path [#1] (\tikzinputsegmentfirst)
        .. controls
        (\tikzinputsegmentsupporta) and (\tikzinputsegmentsupportb)
        ..
        (\tikzinputsegmentlast);
      },
      closepath code={
        \path [#1]
        (\tikzinputsegmentfirst) -- (\tikzinputsegmentlast);
      },
    },
  },
  mid arrow/.style={postaction={decorate,decoration={
        markings,
        mark=at position 0.6 with {\arrow[#1]{stealth}} 
      }}},
}
\usetikzlibrary{arrows}
\usetikzlibrary{trees}

\usetikzlibrary{matrix}
\usetikzlibrary{patterns}
\usetikzlibrary{shadings} 
\date{}
\allowdisplaybreaks[4] \footskip=15pt
\renewcommand{\uppercasenonmath}[1]{}

\numberwithin{equation}{section} \theoremstyle{plain}
\newtheorem{theorem}{Theorem}[section]
\newtheorem{corollary}[theorem]{Corollary}
\newtheorem{lemma}[theorem]{Lemma}

\theoremstyle{definition}
\newtheorem{definition}[theorem]{Definition}
\newtheorem{example}[theorem]{Example}
\newtheorem{remark}[theorem]{Remark}

\newtheorem*{ack*}{ACKNOWLEDGEMENTS}

\newcommand{\oo}{\otimes}

\newcommand{\pf}{\noindent\begin {proof}}
\newcommand{\epf}{\end{proof}}

\newcommand{\ra}{\rightarrow}

\newcommand{\Hom}{\mbox{\rm Hom}}
\newcommand{\Tor}{\mbox{\rm Tor}}

\newcommand{\coker}{\mbox{\rm Coker}}
\newcommand{\ke}{\mbox{\rm Ker}}

	\newcommand{\ModL}{\mathrm{Mod}\text{-}\Lambda}

	\newcommand{\ModA}{\mathrm{Mod}\text{-}A}
	\newcommand{\ModB}{\mathrm{Mod}\text{-}B}
	\newcommand{\ModR}{\mathrm{Mod}\text{-}R}
	\newcommand{\sMod}[1]{\mathrm{mod}\text{-}{#1}}
	\newcommand{\ModAB}{\mathrm{Mod}\text{-}A\times B}
	\newcommand{\modAB}{\mathrm{mod}\text{-}A\times B}
	
	\newcommand{\Modg}{\mathrm{Mod}\text{-}\Gamma}
	\newcommand{\ModG}{\mathrm{Mod}\text{-}\Gamma\ltimes_{\theta}M}
	\newcommand{\modG}{\mathrm{mod}\text{-}\Gamma\ltimes_{\theta}M}
	\newcommand{\modR}{\mathrm{mod}\text{-}R}
	\newcommand{\modA}{\mathrm{mod}\text{-}A}
	\newcommand{\modB}{\mathrm{mod}\text{-}B}
	\newcommand{\modg}{\mathrm{mod}\text{-}\Gamma}
	\newcommand{\modL}{\mathrm{mod}\text{-}\Lambda}
	
	\def\add{\mathop{\rm add}\nolimits}
	\def\dim{\mathop{\rm ext.dim}\nolimits}
	\def\der{\mathop{\rm Rouq.dim}\nolimits}
	\def\IT{\mathop{\rm IT.dist}\nolimits}
    
\begin{document}
\begin{center}
{
{\bf\large
Three homological invariants under cleft extensions
\footnote{Yajun Ma was supported by the Gansu Province Education Science and Technology Innovation Project (Grant No. 2024A-040);
\\
{\color{white}\hspace{5mm}}
Junling Zheng was partly supported by
the National Natural Science Foundation of China (Grant No. 12001508);
\\
{\color{white}\hspace{5mm}}
Yu-Zhe Liu was supported by National Natural Science Foundation of China (Grant Nos. 12561008, 12401042), Guizhou Provincial Basic Research Program (Natural Science) (Grant No. ZD[2025]085, ZK[2024]YiBan066), and Scientic Research Foundation of Guizhou University (Grant Nos. [2022]53, [2022]65, [2023]16).
}}

\vspace{0.5cm}   Yajun Ma, Junling Zheng\footnote{{\color{white}\hspace{1mm}} Corresponding author.}, Yu-Zhe Liu}

\end{center}

$$\bf  Abstract$$
\leftskip0truemm \rightskip0truemm \noindent
In this paper, we investigate the behavior of Igusa-Todorov distances, extension and Rouquier dimensions under cleft extensions of abelian categories.
 We apply our results to Morita context rings, trivial extension rings, tensor rings and arrow removals.
\leftskip10truemm \rightskip10truemm \noindent
\\[2mm]

\leftskip0truemm \rightskip0truemm
\section { \bf Introduction}
The dimension for a triangulated category was introduced by Rouquier (\cite{Rouquier})  following the idea of  Bondal and Van den Bergh' work in \cite{BV},
which measures how quickly such a category can be built from single object. The dimension of a triangulated category has played an important role in representation theory, and has been studied by many scholars.
For example, it can be used to compute the representation dimension of Artin algebras (\cite{Opp,Rouquier1}).
In recent years, there has been tremendous interest in investigating the bounds for the dimension of a triangulated category.
Oppermann and {\v{S}}{\v{t}}ov{\'{\i}}{\v{c}}ek proved that all proper thick subcategories of  the bounded derived category  of finitely generated modules over a Noetherian algebra containing  the perfect complexes have infinite dimension (\cite{OS});  Psaroudakis investigated bounds for the dimension of a triangulated category in a recollement situation by applying \cite[Lemma 3.5]{Rouquier} to the case of semiorthogonal decompositions (\cite{P});  Zheng and Huang gave some bounds for the dimension of a bounded derived category in term of projective dimensions of certain simple modules as well as radical layer length (\cite{Z,Zheng-huang2020,Zheng-huang2022}).

 Inspired by the dimension $\der \mathscr{T}$ of a triangulated category $\mathscr{T}$, the extension dimension $\dim \mathscr{A}$ of an abelian category $\mathscr{A}$ was introduced by Beligiannis in \cite{Be}.
This dimension measures how far an algebra is from being of finite representation type; specially, an Artin algebra $\Lambda$ is representation-finite if and only if the extension dimension of $\modL$ is zero by \cite{Be}.
Many authors have investigated the extension dimension under different situations.
For instance, Zhang and Zheng studied the extension dimension under derived and stable equivalences (\cite{ZZ}).
 It is worth noting that, by \cite{Zheng-huang2022}, we can see that the extension dimension and Rouquier dimension  are closely related to the $n$-Igusa-Todorov algebras introduced by Wei (\cite{W}).
It is well known that Igusa-Todorov algebras satisfy the finitistic dimension conjecture (see \cite[Theorem 1.1]{W}) and are invariant under derived equivalences (see \cite{Wei2018,Wu-Wei2022,Z1}).
Although Igusa-Todorov algebras encompass many classes of Artin algebras (see \cite{W,xi2004finitistic,xi2006finitistic,xi2008finitistic}), unfortunately not all Artin algebras are Igusa-Todorov algebras (\cite{Conde,Z1}).
In order to measure how far an algebra is from being an Igusa-Todorov algebra, Zheng introduced the concept of Igusa-Todorov distance of Artin algebras (\cite{Z1}),
 and studied its behavior under singular equivalences and recollements of derived module categories (\cite{ZZ1}).

A cleft extension of an abelian category $\mathscr{B}$, introduced by Beligiannis (\cite{Be1}), is an abelian category $\mathscr{A}$ together with functors
$$\scalebox{0.85}{\xymatrixcolsep{3pc}\xymatrix{
		\mathscr{B}\ar[rr]^{\mathsf{i}} && \mathscr{A}\ar[rr]^{\sf e} && \mathscr{B} \ar@/_2.0pc/[ll]_{\sf l} } }$$
such that the functor $\sf e$ is faithful and exact, and admits a left adjoint $\sf l$, and there is a natural isomorphism $\mathsf{el}\simeq
\mathrm{Id}_{\mathscr{B}}.$
We denote this cleft extension by $(\mathscr{B},\mathscr{A},\mathsf{i},\mathsf{e},\mathsf{l})$,
which  gives rise to an endofunctor $\mathsf{F}:\mathscr{B}\ra \mathscr{B}$  ( see Section \ref{Section2.2}); we emphasize that the  endofunctor is of importance, as imposing homological conditions on it imply properties for the rest of the functors.
It should be noted that there is an abundance of constructions that are of great interest in representation theory and occur as a cleft extension of  abelian categories.
For instance, Morita context rings, trivial extensions, tensor rings and $\theta$-extensions \cite{Kostas,Ma}.
Hence cleft extensions of abelian categories provide the proper conceptual framework for the study of many problems in several different contexts.
For example, Green, Psaroudakis and Solberg investigated the behavior of the finitistic dimension in a cleft extension under certain condition (\cite{GP1}); Kostas studied Gorensteinness, Gorenstein projective modules, and singularity categories in the cleft extensions of rings (\cite{Kostas}).

Building on this, the main purpose of this paper is to investigate the behavior of Igusa-Todorov distances, extension dimensions, and dimensions of bounded derived categories under cleft extensions of abelian categories.

To this end, we first introduce the Igusa-Todorov distance, denoted by $\IT (\mathscr{A})$, for an abelian category $\mathscr{A}$. This is a generalization of the Igusa-Todorov distance of an Artin algebra introduced by Zheng (\cite{Z1}); see Definition \ref{def:IT distance}. We then present the following theorem, which is  listed as Theorem \ref{thm:IT}.

\begin{theorem}\label{thm1}
	Let $(\mathscr{B},\mathscr{A},\mathsf{i},\mathsf{e},\mathsf{l})$ be a cleft extension of abelian categories with enough projectives such that $\mathsf{l}$ is exact and $\mathsf{e}$ preserves projectives. If the induced endofunctor $\mathsf{F}$ is nilpotent and $\mathsf{Proj}\mathscr{B}=\add (P)$ for some projective object $P$ in $\mathscr{B}$, then
	$$\IT (\mathscr{B})\leq \IT (\mathscr{A})\leq n(\IT (\mathscr{B})+1)-1,$$
	where $n$ is a positive integer such that $\mathsf{F}^{n}=0$.
\end{theorem}

 There is a dual notion to cleft extensions, which is called cleft coextensions (see Definition \ref{def:cleft coextension}). Let $(\mathscr{B},\mathscr{A},\mathsf{i},\mathsf{e},\mathsf{r})$ be a cleft coextension of abelian categories. This gives rise to an endofunctor $\mathsf{F}':\mathscr{B}\ra \mathscr{B}$ ( see Section \ref{Section 2.3}).
 From \cite[Remark 2.5]{GP1}, we know that the cleft coextension $(\mathscr{B},\mathscr{A},\mathsf{i},\mathsf{e},\mathsf{r})$ is the lower part of a cleft extension $(\mathscr{B},\mathscr{A},\mathsf{i},\mathsf{e},\mathsf{l})$
 if and only if $\mathsf{F}'$ admits a left adjoint $\mathsf{F}$, where $\mathsf{F}$ is the induced endofunctor associated to $(\mathscr{B},\mathscr{A},\mathsf{i},\mathsf{e},\mathsf{l})$.
Let $\Lambda$ be a Noetherian algebra. Note that the category of finitely generated right $\Lambda$-modules, denoted by $\modL$, does not have enough injectives generally.
So in order to discuss the extension dimension for cleft extensions of Noetherian algebras, we investigate it within the framework of cleft coextensions; see Lemma \ref{main lem0} and Theorem \ref{main thm}.

\begin{theorem}\label{thm2}
	Let $(\mathscr{B},\mathscr{A},\mathsf{i},\mathsf{e},\mathsf{r})$ be a cleft coextension of abelian categories with enough projectives and suppose $\mathsf{Proj}\mathscr{A}=\mathrm{add}(P)$ for some projective object $P$ in $\mathscr{A}$. If $\mathsf{r}$ is exact and the induced endofunctor $\mathsf{F}'$ is nilpotent,
	then $$\dim \mathscr{B}\leq \dim \mathscr{A}\leq n(\dim \mathscr{B}+1)-1,$$
	where $n$ is a positive integer such that $\mathsf{F}'^{n}=0$.
\end{theorem}

Finally we consider the dimension of bounded derived module categories involved in a cleft extension of module categories when the induced endofunctor $\mathsf{F}$ is left perfect and nilpotent (see Definition \ref{def:left perfect}).
For details, we refer to Theorem \ref{thm:dimension of triangulated category}.
\begin{theorem}\label{thm3}	
	Let $A$ and $B$ be two Noetherian algebras such that
 $(\ModB,\ModA,\mathsf{i},\mathsf{e},\mathsf{l})$ is a cleft extension of module categories.
	If the induced endofunctor $\mathsf{F}$ is left perfect and nilpotent,
	then $$\der \mathbf{D}^{b}(\modB)\leq \der \mathbf{D}^{b}(\modA)\leq n(\der \mathbf{D}^{b}(\modB)+1)-1,$$
	where $n$ is a positive integer such that $\mathsf{F}^{n}=0.$
\end{theorem}

As shown in \cite{GP1},  arrow removal on an admissible path algebra yields a cleft extension with specific homological properties, we get the following result.

\begin{theorem}
	Let $\Lambda=kQ/I$ be an admissible quotient of a path algebra $kQ$ over a field $k.$
	Suppose that there are arrows $a_{i}:\upsilon_{e_{i}}\ra \upsilon_{f_{i}}$ in $Q$ for $i=1,2,\cdots, t$ which do not occur in a set of minimal generators of $I$ in $kQ$ and $\Hom_{\Lambda}(e_{i}\Lambda,f_{i}\Lambda)=0$ for all $i,j$ in $\{1,2,\cdots,t\}$.
	Let $\Gamma=\Lambda/\Lambda\{\overline{a_{i}}\}_{i=1}^{t}\Lambda$.
	Then the following hold:
	\begin{enumerate}
		\item $\IT (\Gamma)\leq\IT(\Lambda)\leq 2\IT(\Gamma)+1.$
		\item $\dim (\Gamma)\leq\dim\Lambda\leq 2\dim(\Gamma)+1.$
		\item $\der \mathbf{D}^{b}(\modg)\leq\der\mathbf{D}^{b}(\modL)\leq 2\der \mathbf{D}^{b}(\modg)+1.$
	\end{enumerate}
\end{theorem}

The paper is structured as follows. In Section 2, we introduce essential notations and provide the foundational definitions of cleft extensions and cleft coextensions of abelian categories, along with their basic homological properties necessary for subsequent proofs. Section 3 is dedicated to the proof of Theorem 1.1, while Section 4 presents the proof of Theorem 1.2. In Section 5, we give the proof of Theorem 1.3. Finally, we explore applications of these results to Morita context rings, $\theta$-extensions of rings, and tensor rings, demonstrating the practical implications of our theoretical framework.

\section { \bf Preliminaries}
\subsection{Conventions}
All rings considered in this paper are nonzero associative rings with identity, and consequently, all modules are unitary.
For a ring $R$, we work usually with right $R$-modules and write $\ModR$ for the category of right $R$-modules.
The full subcategory of finitely generated right $R$-modules is denoted by $\modR$.
By a module over a  Noetherian algebra $\Lambda$, we mean a finitely generated $\Lambda$-module.

Let $\mathscr{A}$ be an abelian category with enough projectives.
We denote by $\mathsf{Proj} \mathscr{A}$ the full subcategory of $\mathscr{A}$ consisting of projective objects.

A \emph{subcategory} of $\mathscr{A}$ is always assumed to be a full and additive subcategory  which is closed under isomorphisms, and the word \emph{functor}
will mean an additive functor between additive categories.
For a subclass $\mathcal{U}$ of $\mathscr{A}$, we use $\add (\mathcal{U})$ to
denote the subcategory of $\mathscr{A}$ consisting of
direct summands of finite direct sums of objects in $\mathcal{U}$.
When $\mathcal{U}$ consists of single object $U$, we write $\add (U)$ for $\add (\mathcal{U}).$

Let $\mathbf{C}(\mathscr{A})$ be the category of all complexes over $\mathscr{A}$ with chain maps, and the corresponding homotopy category
is denoted by $\mathbf{K}(\mathscr{A})$.
The derived category of $\mathscr{A}$ is denoted by $\mathbf{D}(\mathscr{A}),$ which is the localization of $\mathbf{K}(\mathscr{A})$ at all quasi-isomorphisms.
The full subcategory of $\mathbf{D}(\mathscr{A})$ consisting of complexes isomorphic to bounded complexes is denoted by $\mathbf{D}^{b}(\mathscr{A})$.

Unless otherwise stated, abelian categories are assumed to
	have enough projectives.

\subsection{Cleft extensions  of abelian categories}\label{Section2.2}
We first recall the definition of a cleft extension, which was introduced by Beligiannis (\cite{Be1,Be2}).

\begin{definition}(\cite[Definition 2.1]{Be1})
	A {\bf cleft extension} of an abelian category $\mathscr{B}$ is an abelian category $\mathscr{A}$ together with functors
	$$\scalebox{0.85}{\xymatrixcolsep{3pc}\xymatrix{
			\mathscr{B}\ar[rr]^{\mathsf{i}} && \mathscr{A}\ar[rr]^{\sf e} && \mathscr{B} \ar@/_2.0pc/[ll]_{\sf l} } }$$
	satisfying the following:
	\begin{enumerate}
		\item The functor $\mathsf{e}$ is faithful  and exact.
		\item The pair $(\mathsf{l},\mathsf{e})$ is an adjoint pair.
		\item There is a natural isomorphism $\varphi: \mathsf{ei}\ra \mathrm{Id}_{\mathscr{B}}$.
	\end{enumerate}
\end{definition}

In the following, we denote a cleft extension by $(\mathscr{B},\mathscr{A},\mathsf{i},\mathsf{e},\mathsf{l})$.
The above data give rise to more structure information.
For instance, it follows from \cite[Lemma 2.2]{GP1} that $\mathsf{i}$ is fully faithful and exact. Moreover, there is a functor $\mathsf{q}:\mathscr{A}\ra \mathscr{B}$ such that $(\mathsf{q},\mathsf{i})$ forms an adjoint pair.

We now explain the certain endofunctors on the categories $\mathscr{A}$ and $\mathscr{B}$, which are important for our result. Denote by $\nu:\rm Id_{\mathscr{B}}\ra \mathsf{el}$ the unit and by $\mu: \mathsf{le}\ra \mathrm{Id}_{\mathscr{A}}$ the counit of the adjoint pair $(\mathsf{l},\mathsf{e})$.
The unit and counit satisfy the following relations
$$\mathrm{Id}_{\mathsf{l}(B)}=\mu_{\mathsf{l}(B)}\mathsf{l}(\nu_{B})$$
and
$$\mathrm{Id}_{\mathsf{e}(A)}=\mathsf{e}(\mu_{A})\nu_{\mathsf{e}(A)}$$
for all $A$ in $\mathscr{A}$ and $B$ in $\mathscr{B}$.
It follows that $\mathsf{e}(\mu_{A})$ is a (split) epimorphism. Since $\sf e$ is faithful, this implies that $\mu_{A}: \mathsf{le}(A)\ra A$ is an epimorphism for every $A\in \mathscr{A}$.
Hence we have the following exact sequence
$$0\ra \ke \mu_{A}\ra \mathsf{le}(A)\xrightarrow{\mu_{A}} A\ra 0$$
for every $A\in \mathscr{A}.$
The assignment $A\mapsto\ke \mu_{A}$ defines an endofunctor $\mathsf{G}:\mathscr{A}\ra \mathscr{A}$.
Consider now an object $B$ in $\mathscr{B}$.
Then we have the short exact sequence in $\mathscr{A}$:
\begin{equation}\label{se.G}
0\ra \mathsf{G}(\mathsf{i}(B))\ra \mathsf{le}(\mathsf{i}(B))\xrightarrow{\mu_{\mathsf{i}(B)}} \mathsf{i}(B)\ra 0.
\end{equation}
We denote by $\mathsf{F}(B)=\mathsf{e}(\mathsf{G}(\mathsf{i}(B))).$
Viewing the natural isomorphism $\mathsf{ei}(B)\xrightarrow{\sim} B$ as an identification and applying the exact functor $\mathsf{e}:\mathscr{A}\ra \mathscr{B}$ to (\ref{se.G}), we obtain the exact sequence:
\begin{equation}
	\label{se.F}
	0\ra \mathsf{F}(B)\ra \mathsf{el}(B)\xrightarrow{}B\ra 0.
\end{equation}
The assignment $B\mapsto \mathsf{F}(B)$ defines an endofunctor $\mathsf{F}:\mathscr{B}\ra \mathscr{B}.$

 In the following lemma, we consider $\mathsf{F}^{n}$ and $\mathsf{G}^{n}$ and collect some results from \cite{Be1,GP1} which are used in the proof of our main result. We mention that $\mathsf{F}^{0}$ and $\mathsf{G}^{0}$ are identity functors.

\begin{lemma}\label{lem of cleft extension}
	Let $(\mathscr{B},\mathscr{A},\mathsf{i},\mathsf{e},\mathsf{l})$ be a cleft extension of abelian categories. The following statements hold:
	
	\begin{enumerate}
		\item[$(1)$] Let $n\geq 1$. Then $\mathsf{F}^{n}=0$ if and only if $\mathsf{G}^{n}=0.$
		
		\item[$(2)$] For any $n\geq 1$ and any $A\in \mathscr{A}$, there is an exact sequence
		$$0\ra \mathsf{G}^{n}(A)\ra \mathsf{l}\mathsf{F}^{n-1}\mathsf{e}(A)\ra \mathsf{G}^{n-1}(A)\ra 0.$$
	\end{enumerate}
\end{lemma}
\begin{proof}
	(1) This follows directly from \cite[Lemma 2.4]{GP1}.
	
	(2) By the definition of $\sf G$, there is an exact sequence
	\begin{align}
		\label{lem:exact sequence}
		0\ra \mathsf{G}^{n}(A)\ra  \mathsf{le}\mathsf{G}^{n-1}(A)\ra \mathsf{G}^{n-1}(A)\ra  0
	\end{align}
	for any $A\in \mathscr{A}$ and $n\geq 1$.
	 Since $\mathsf{le}\mathsf{G}^{n}\simeq \mathsf{l}\mathsf{F}^{n}\mathsf{e}$ by \cite[Lemma 2.4(1)]{GP1} or \cite[Subsection 5.2]{Be1} for any $n\geq 1$, it follows that the desired exact sequence is obtained by replacing $\mathsf{le}\mathsf{G}^{n-1}(A)$ with $\mathsf{l}\mathsf{F}^{n-1}\mathsf{e}(A)$ in \eqref{lem:exact sequence}.
\end{proof}

\subsection{Cleft coextensions  of abelian categories}\label{Section 2.3}
There is a dual notion to cleft extensions, which is called cleft coextensions.
We first recall its definition.
\begin{definition}(\cite[Definition 2.1]{Be1})\label{def:cleft coextension}
	A {\bf cleft coextension} of an abelian category $\mathscr{B}$ is an abelian category $\mathscr{A}$ together with functors
	$$\scalebox{0.85}{\xymatrixcolsep{3pc}\xymatrix{
			\mathscr{B}\ar[rr]^{\sf i} && \mathscr{A}\ar[rr]^{\sf e} && \mathscr{B} \ar@/^2.0pc/[ll]^{\sf r} } }$$
	satisfying the following:
	\begin{enumerate}
		\item The functor $\sf e$ is faithful and exact.
		\item The pair $(\sf e,\sf r)$ is an adjoint pair.
		\item There is a natural isomorphism $\varphi:{\sf ei}\ra \mathrm{Id}_{\mathscr{B}.}$
	\end{enumerate}
\end{definition}


In the following, we denote a cleft coextension by $(\mathscr{B},\mathscr{A},\mathsf{i},\mathsf{e},\mathsf{r})$. We now describe the certain endofunctors on the categories $\mathscr{A}$ and $\mathscr{B}$, which are important for the proof of Theorem \ref{main thm}. Denote by $\nu':\rm Id_{\mathscr{A}}\ra \sf re$ the unit and by $\mu':{\sf er}\ra \mathrm{Id}_{\mathscr{B}}$ the counit of the adjoint pair $(\sf e,r)$.
The unit and counit satisfy the following relations
$$ \mathrm{Id}_{\mathsf{e}(A)}=\mu'_{\mathsf{e}(A)}\mathsf{e}(\nu'_{A})$$
and
$$ \mathrm{Id}_{\mathsf{r}(B)}=\mathsf{r}(\mu'_{B})\nu'_{\mathsf{r}(B)})$$
for all $A$ in $\mathscr{A}$ and $B$ in $\mathscr{B}$.
It follows that $\mathsf{e}(\nu'_{A})$ is a (split) monomorphism. Since $\sf e$ is faithful, this implies that $\nu'_{A}: A\ra \mathsf{re}(A)$ is a monomorphism for every $A\in \mathscr{A}$.
Hence we have the following exact sequence
$$0\ra A\xrightarrow{\nu'_{A}} \mathsf{re}(A)\ra \coker\nu'_{A} \ra 0$$
for every $A\in \mathscr{A}.$
The assignment $A\mapsto\coker \nu'_{A}$ defines an endofunctor $\mathsf{G}':\mathscr{A}\ra \mathscr{A}$.
Consider now an object $B$ in $\mathscr{B}$.
Then we have the short exact sequence in $\mathscr{A}$:
\begin{equation}\label{se.G1}
0\ra \mathsf{i}(B)\xrightarrow{\nu'_{\mathsf{i}(B)}}\mathsf{re}(\mathsf{i}(B)) \ra \mathsf{G}'(\mathsf{i}(B)) \ra 0.
\end{equation}
We denote by $\mathsf{F}'(B)=\mathsf{e}(\mathsf{G}'(\mathsf{i}(B))).$
Viewing the natural isomorphism $\mathsf{ei}(B)\xrightarrow{\sim} B$ as an identification and applying the exact functor $\mathsf{e}:\mathscr{A}\ra \mathscr{B}$ to (\ref{se.G1}), we get the exact sequence:
$$0\ra B\xrightarrow{} \mathsf{er}(B)\ra \mathsf{F}'(B)\ra 0.$$
The assignment $B\mapsto \mathsf{F}'(B)$ defines an endofunctor $\mathsf{F}':\mathscr{B}\ra \mathscr{B}.$

\begin{remark}
(\cite[Remark 2.7]{Be1}, \cite[Remark 2.5]{GP1}) A cleft extension of an abelian category $(\mathscr{B},\mathscr{A},\mathsf{i},\mathsf{e},\mathsf{l})$ is the upper part of a cleft coextension $(\mathscr{B},\mathscr{A},\mathsf{i},\mathsf{e},\mathsf{r})$ if and only if the endofunctor $\mathsf{F}$ appearing in \eqref{se.F} has a right adjoint $\mathsf{F}'$.
	In this case, $\mathsf{F}'$ is the endofunctor of $\mathscr{B}$ induced by the cleft coextension $(\mathscr{B},\mathscr{A},\mathsf{i},\mathsf{e},\mathsf{r})$.
\end{remark}

Every property of a cleft coextension of an abelian category is dual to that of a cleft extension of an abelian category.
The following lemma will be used in the description of extension dimension. By convention, given an endofunctor $\sf{F}$, we set $\sf{F}^{0}$ to be the identity functor.

\begin{lemma} \label{lem of cleft coextension}
	Let $(\mathscr{B},\mathscr{A},\mathsf{i},\mathsf{e},\mathsf{r})$ be a cleft coextension of abelian categories. The following statements hold:
	
	\begin{enumerate}
		\item[$(1)$] Let $n\geq 1$. Then $\mathsf{F}'^{n}=0$ if and only if $\mathsf{G}'^{n}=0.$
		
		\item[$(2)$] For any $n\geq 1$ and any $A\in \mathscr{A}$, there is an exact sequence
		$$0\ra \mathsf{G}'^{n-1}(A)\ra \mathsf{r}\mathsf{F}'^{n-1}\mathsf{e}(A)\ra \mathsf{G}'^{n}(A)\ra 0.$$
	\end{enumerate}
\end{lemma}
\begin{proof}
	(1) The proof is dual to that of \cite[Lemma 2.4]{GP1}.
	
	(2) By the definition of $\mathsf{G}'$, there is an exact sequence
	$$0\ra \mathsf{G}'^{n-1}(A)\ra  \mathsf{re}\mathsf{G}'^{n-1}(A)\ra \mathsf{G}'^{n}(A)\ra  0$$
	for any $A\in \mathscr{A}$ and $n\geq 1$.
	Since $\mathsf{e}\mathsf{G}'^{n}\simeq \mathsf{F}'^{n}\mathsf{e}$ by the dual of \cite[Lemma 2.4(1)]{GP1} for any $n\geq 1$, it follows that we obtain the following exact sequence
	$$0\ra \mathsf{G}'^{n-1}(A)\ra \mathsf{r}\mathsf{F}'^{n-1}\mathsf{e}(A)\ra \mathsf{G}'^{n}(A)\ra 0,$$
	which completes the proof.
\end{proof}
\subsection{Cleft extension of module categories}
We now turn our attention to cleft extensions of module categories, which can also be viewed as cleft coextensions of module categories by the following lemma.
\begin{lemma}\label{extension-coextension}
	{\rm (}\cite[Proposition 4.19]{Kostas}{\rm)} A cleft extension  $(\Modg,\ModL,\mathsf{i},\mathsf{e},\mathsf{l})$ of module categories is the upper part of a cleft coextension. Moreover, if $\mathsf{F}$ is  nilpotent, then so is
	$\mathsf{F}'$.
\end{lemma}
We now give some examples of cleft extensions of module categories which will be used throughout the paper.
\begin{example}\label{ex:Morita ring} (\cite[Example 5.4]{Kostas})
	Let $A$ and $B$ be two rings, and let $_{B}M_{A}$ and $_{A}N_{B}$ be bimodules.
	Consider the Morita context ring $$\Lambda =\left(\begin{matrix}  A & {_{A}}N_{B}\\  {_{B}}M_{A} & B \\\end{matrix}\right)$$
	with zero bimodule homomorphisms.
	
	We recall from \cite{eg82} that the module categories over the Morita context ring $\Lambda$ is equivalent to a category whose objects are tuples $(X,Y,f,g)$, where $X\in \text{Mod-}A$, $Y\in \text{Mod-}B$, $f\in \text{Hom}_{B}(X\oo_{A}N,Y)$ and $g\in \text{Hom}_{A}(Y\oo_{B}M,X)$.
	Thus we have a cleft extension of module categories, which is also a  upper part of a cleft coextension by Lemma \ref{extension-coextension}
	$$\scalebox{0.85}{\xymatrixcolsep{3pc}\xymatrix{
			\ModAB\ar[rr]^{\sf i} && \ModL\ar[rr]^{\sf e}  && \ModAB. \ar@/_2.0pc/[ll]_{\sf l} \ar@/^2.0pc/[ll]^{\sf r}} }$$
	These functors are given as follows:
	\begin{enumerate}
		\item $\mathsf{i}(X,Y)=(X,Y,0,0).$
		\item $\mathsf{e}(X,Y,f,g)=(X,Y).$
		\item $\mathsf{l}(X,Y)=(X,X\oo_{A}N,0,1)\oplus (Y\oo_{B}M,Y,1,0)$.
		\item $\mathsf{r}(X,Y)=(X,\Hom_{A}(M,X),0,\varepsilon'_{X})\oplus (\Hom_{B}(N,Y),Y,\varepsilon_{Y},0)$, where
		\begin{center}
			$\varepsilon':\Hom_{A}(M,-)\oo_{B}M\ra \mathrm{Id}_{\ModA}~\hspace{0.5cm}\text{and}\hspace{0.5cm} \varepsilon_{}:\Hom_{B}(N,-)\oo_{A}N\ra \mathrm{Id}_{\ModB}$
		\end{center}
		are the counits of adjoint pairs $(-\oo_{B}M,\Hom_{A}(M,-))$ and $(-\oo_{A}N,\Hom_{B}(N,-))$ respectively.
		\item $\mathsf{F}(X,Y)=(Y\oo_{B}M,X\oo_{A}N)$.
		\item $\mathsf{F}'(X,Y)=(\Hom_{B}(N,Y),\Hom_{A}(M,X))$.
	\end{enumerate}
\end{example}

\begin{example}\label{ex:extension}(\cite[Section 5.2]{Kostas})
	Let $\Gamma$ be a ring, $M$ a $\Gamma$-$\Gamma$-bimodule, and $\theta:M\oo_{\Gamma}M\ra M$ an associative $\Gamma$-bimodule homomorphism, i.e., a bimodule homomorphism such that $\theta(\theta\otimes \mathrm{Id}_{M})=\theta(\mathrm{Id}_{M}\otimes \theta)$.
	Then $\theta$-\emph{extension} of $\Gamma$ by $M$, denoted by $\Gamma\ltimes_{\theta}M$, is defined to be the ring with underlying group $\Gamma\oplus M$ and the multiplication given as follows:
	$$(\gamma,m)\cdot(\gamma',m')=(\gamma\gamma',\gamma m'+m\gamma'+\theta(m\oo m')),$$
	for $\gamma,\gamma'\in\Gamma$ and $m,m'\in M$.
	
	Consider the $\theta$-extension $\Gamma\ltimes_{\theta}M$ and ring homomorpism $f:\Gamma\ltimes_{\theta}M\ra \Gamma$ and $g:\Gamma\ra \Gamma\ltimes_{\theta}M$ given by $(\gamma,m)\mapsto\gamma$ and $\gamma\mapsto(\gamma,0)$ respectively.
	Denote by $\mathbf{Z}:\Modg\ra \ModG$ and $\mathbf{U}:\ModG\ra \Modg$ the restriction functors induced by $f$ and $g$ respectively.
	Then we have the cleft extension of module categories, which is also a upper part of a cleft coextension by Lemma \ref{extension-coextension}:
	
	$$\scalebox{0.85}{\xymatrixcolsep{3pc}\xymatrix{
			\Modg\ar[rr]^{\mathbf{Z}} && \ModG\ar[rr]^{\mathbf{U}}  && \Modg. \ar@/_2.0pc/[ll]_{\mathbf{T}=-\oo_{\Gamma}\Gamma\ltimes_{\theta}M} \ar@/^2.0pc/[ll]^{\mathbf{H}=\mathrm{Hom}_{\Gamma}(\Gamma\ltimes_{\theta}M,-)} } }$$
 A categorical description of the functors of this cleft extension can be found in \cite{Ma}.
\end{example}

\begin{remark} (\cite[Theorem 2.6, Proposition 3.3]{Be1} or \cite[Proposition 6.9]{KP}) \label{remark:fg}
	Let $(\Modg,\ModL,\mathsf{i},\mathsf{e},\mathsf{l})$ be a cleft extension of module categories.
	Then there is a $\Gamma$-bimodule $M$, an associative $\Gamma$-bimodule homomorphism $\theta:M\oo_{\Gamma}M\ra M$, and an equivalence $\ModL\xrightarrow{\sim} \ModG$
	making the following diagram commutative:
	$$\scalebox{0.85}{\xymatrixcolsep{3pc}\xymatrix{
			\Modg\ar[rr]^{\mathsf{i}}\ar@{=}[dd] && \ModL\ar[rr]^{\mathsf{e}}\ar[dd]_{\simeq} \ar@/_2.0pc/[ll]_{\mathsf{q}}  && \Modg \ar@/_2.0pc/[ll]_{\mathsf{l}} \ar@{=}[dd] \\
			&&&\\
			\Modg\ar[rr]^{\mathbf{Z}} && \ModG\ar[rr]^{\mathbf{U}} \ar@/_2.0pc/[ll]_{\mathbf{C}} && \Modg \ar@/_2.0pc/[ll]_{\mathbf{T}}
	} }$$
	where the bottom diagram is the cleft extension associated with $\Gamma\ltimes_{\theta}M$.
	If moreover $\Gamma$ and $\Lambda$ are module finite over a commutative Noetherian ring, i.e., they are Noetherian algebras, then these functors restrict to the level of finitely generated modules by \cite[Lemma 5.12]{Kostas} and \cite[Proposition 18.32]{Lam}.
	Consequently, we obtain the following cleft extension
	$$\scalebox{0.85}{\xymatrixcolsep{3pc}\xymatrix{
			\modg\ar[rr]^{\mathsf{i}} && \modL\ar[rr]^{\mathsf{e}} \ar@/_2.0pc/[ll]_{\mathsf{q}} && \modg, \ar@/_2.0pc/[ll]_{\mathsf{l}} } }$$
	which is also a upper part of a cleft coextension by Lemma \ref{extension-coextension}.
\end{remark}

\section{Igusa-Todorov distances}
In this section we introduce the notion of Igusa-Todorov distance of an abelian category, which is a generalization of Igusa-Todorov distance of an Artin algebra given in \cite{Z1}.
To this end, it is convenient to introduce the following notion.

\begin{definition}
Let $\mathscr{A}$ be an abelian category and $n$ a positive integer.
If there is an exact sequence
$$0\ra K\ra P_{n-1}\ra P_{n-2}\ra \cdots \ra P_{0}\ra X\ra 0$$
in $\mathscr{A}$ with all $P_{i}$ projective for $0\leq i\leq n-1$, then $K$ is called an $n$-syzygy object of $X$. We denote by $\Omega^{n}(\mathscr{A})$ the subcategory of $\mathscr{A}$ consisting of all $n$-syzygy objects.
\end{definition}

\begin{remark}
Given an object $X$ in an abelian category $\mathscr{A}$, its $n$-syzygy objects are not unique up to isomorphism, but are unique up to projective summands.
\end{remark}

\begin{definition}\label{def-IT-cat}
	Let $m$ and $n$ be nonnegative integers. An abelian category $\mathscr{A}$
	with enough projective objects is called an {\bf $(m,n)$-Igusa-Todorov category} if there is an object $U\in \mathscr{A}$ such that for any $A\in \Omega^{n}(\mathscr{A})$, there is an exact sequence
	$$0\ra U_{m}\ra U_{m-1}\ra\cdots\ra U_{1}\ra U_{0}\ra A\ra 0,$$
	with $U_{i}\in \add (U)$ for each $0\leq i\leq m$.
\end{definition}

\begin{definition}\label{def:IT distance}
	Let $\mathscr{A}$ be an abelian category.
	The {\bf Igusa-Todorov distance} of $\mathscr{A}$ is defined as
	$$\IT(\mathscr{A}):=\inf\{m \mid\mathscr{A}\text{ is an}~ (m,n)\text{-Igusa-Todorov category for some nonnegative integer}~n\}.$$
\end{definition}

\begin{remark}
Let $\Lambda$ be an Artin algebra. Then $\modL$ is an $(m,n)$-Igusa-Todorov category if and only if $\Lambda$ is an an $(m,n)$-Igusa-Todorov algebra in the sense of \cite[Definition 2.1]{Z}.
	It follows that the Igusa-Todorov distance of $\modL$ is the Igusa-Todorov distance of $\Lambda$; see \cite[definition 2.27]{Z1}.
\end{remark}

\begin{remark}
Let $\Lambda$ be an Artin algebra and let $n$ be a nonnegative integer.
\begin{enumerate}
\item Recall that $\Lambda$ is called {\bf syzygy-finite} if $\Omega^{n}(\modL)$ is representation-finite for some nonnegative integer $n$; that is, the number of non-isomorphic indecomposable direct summands of modules in $\Omega^{n}(\modL)$ is finite.
    It follows from Definition \ref{def:IT distance} that $\Lambda$ is syzygy-finite if and only if $\IT(\modL)=0.$
\item Recall from \cite[Definition 2.2]{W} that $\Lambda$ is called an {\bf $n$-Igusa-Todorov algebra} if there exists a $\Lambda$-module
$U$ such that for any $\Lambda$-module $X\in \Omega^{n}(\modL)$ there exists an exact sequence $0\ra U_{0}\ra U_{1}\ra X\ra 0$
with $U_{i}\in \mathrm{add}(U)$ for each $0\leq i \leq 1$.
 Furthermore, $\Lambda$ is call an {\bf Igusa-Todorov algebra} if it is an $n$-Igusa-Todorov algebra for some nonnegative integer $n.$
Then it follows from Definition \ref{def:IT distance} that $\Lambda$ is an  Igusa-Todorov algebra  if and only if $\IT(\modL)\leq1.$
\end{enumerate}
\end{remark}

\begin{example}{\rm
		Let $k$ be a field, and let $n\geq 1$ be an integer, and let
		$\Lambda$ be the quantum exterior algebra
		$$\Lambda=k[ X_{1},\cdots, X_{n}]/(X_{i}^{2},\{X_{i}X_{j}-q_{ij}X_{j}X_{i}\}_{i<j}),$$
		where $0\neq q_{ij}\in k$ and all the $q_{ij}$ are the roots of unity.
		We know from \cite[Example 2.35]{Z1} that $\IT(\modL)=n-1.$
	}
\end{example}

The following lemma can be found for Artin algebras in \cite[Lemma 3.3]{XuD16}.
For the convenience of the reader, we provide a proof.
\begin{lemma} \label{IT-lem1}
	Let $\mathscr{A}$ be an abelian category. Then
	for any exact sequence
	$$0\longrightarrow X_{1} \longrightarrow  X_{2} \longrightarrow  X_{3} \longrightarrow 0$$
	in $\mathscr{A}$, we have the following exact sequence
	$$0\longrightarrow K\longrightarrow X_{1}\oplus P \longrightarrow X_{2}\longrightarrow 0,$$
	where $P$ is projective in $\mathscr{A}$ and $K$ is a 1-syzygy object of $X_{3}$.
\end{lemma}
\begin{proof}
 Since $\mathscr{A}$ has enough projectives, there is an exact sequence
$$0\ra K\ra P\ra X_{3}\ra 0$$
 with $P$ projective.
 Consider the following pullback diagram:
 $$\xymatrix{&&0\ar[d]&0\ar[d]\\
		&&K\ar@{=}[r]\ar[d]&K\ar[d]\\
		0\ar[r]&X_{1}\ar[r]\ar@{=}[d]&Z\ar[d]\ar[r]&P\ar[d] \ar[r]&0\\
		0\ar[r] &X_{1}\ar[r]&X_{2}\ar[d]\ar[r]\ar[d]&X_{3}\ar[r]\ar[d]&0 \\&&0&0}$$
Since $P$ is projective, the exact sequence
$$0\ra X_{1}\ra Z\ra P\ra 0$$
splits and so $Z\cong X_{1}\oplus P$.
Thus we obtain an exact sequence
$$0\ra K\ra X_{1}\oplus P \ra X_{2}\ra 0$$
with $P$ projective and $K$ a 1-syzygy object of $X_{3}$, as desired.
\end{proof}

\begin{lemma}\label{lem-exact-syzygy}
	Let $\mathscr{A}$ and $\mathscr{B}$ be abelian categories and $n \geq 0$, and let $\mathsf{F}: \mathscr{A}\to \mathscr{B}$ be an exact
	functor preserving projective objects.
Then for any $X\in \mathscr{A}$ and any $n$-syzygy object $K$ of $X$, the image $\mathsf{F}(K)$ is an $n$-syzygy object of $\mathsf{F}(X)$.
\end{lemma}
\begin{proof}
Let $K$ be an $n$-syzygy object $K$ of $X$. Then there is an exact sequence
\begin{equation}\label{preserve syzygy}
0\ra K\ra P_{n-1}\ra P_{n-2}\ra \cdots \ra P_{0}\ra X\ra 0
\end{equation}
in $\mathscr{A}$ with all $P_{i}$ projective for $0\leq i\leq n-1$.
Since $\mathsf{F}$ is exact, applying $\sf F$ to \eqref{preserve syzygy}, we obtain exact sequence
\begin{equation*}
0\ra \mathsf{F}(K)\ra \mathsf{F}(P_{n-1})\ra \mathsf{F}(P_{n-2})\ra \cdots \ra \mathsf{F}(P_{0})\ra \mathsf{F}(X)\ra 0.
\end{equation*}
We mention that each $\mathsf{F}(P_{i})$ is projective for $0\leq i\leq n-1$ as $\mathsf{F}$ preserves projective objects.
It follows that $\mathsf{F}(K)$ is an $n$-syzygy object of $\mathsf{F}(X)$.
\end{proof}

The following lemma is key to the proof of Theorem \ref{thm:IT}.
\begin{lemma}\label{IT-lem2}
	Let $(\mathscr{B},\mathscr{A},\mathsf{i},\mathsf{e},\mathsf{l})$ be a cleft extension of abelian categories such that functor $\mathsf{l}$ is exact.
	Assume that the sequence $0\ra X\ra  \mathsf{l}(B)\ra A\ra 0$  is  exact with $B\in \mathscr{B}$ and $B_{t}$ being a $t$-syzygy object of $B$.
	If there is an exact sequence
	$$0\ra V_{m}\xrightarrow{d_{m}}V_{m-1}\xrightarrow{d_{m-1}}\cdots\ra V_{1}\xrightarrow{d_{1}}V_{0}\xrightarrow{d_{0}}B_{t}\ra 0$$
	for some nonnegative integer $t$,
	then we have the following exact sequence
	$$0\ra X_{t+m}\ra \mathsf{l}(V_{m})\oplus P_{m}\ra\cdots \ra \mathsf{l}(V_{1})\oplus P_{1}\ra \mathsf{l}(V_{0})\oplus P_{0}\ra A_{t}\ra 0,$$
	where each $P_{i}$ belongs to $\mathsf{Proj} \mathscr{A}$ for $0\leq i\leq m$, $A_{t}$ is a $t$-syzygy object of $A$ and $X_{t+m}$ is a $(t+m)$-syzygy object of $X.$
\end{lemma}
\begin{proof}
	We denote $\ke d_{i}$ by $K_{i}$ for $0\leq i\leq m$.
	Consider the exact sequence
	\begin{equation}
		\label{se.3.1}
		0\ra K_{0}\ra V_{0}\ra B_{t}\ra 0.
	\end{equation}
Since $\mathsf{l}$ is exact and preserves projective objects by \cite[Lemma 2.2]{GP1}, it follows from Lemma \ref{lem-exact-syzygy} that $\mathsf{l}(B_{t})$ is a $t$-syzygy object of $\mathsf{l}(B)$.
	Applying $\mathsf{l}$ to exact sequence \eqref{se.3.1}, we get an exact sequence
	$$0\ra \mathsf{l}(K_{0})\ra \mathsf{l}(V_{0})\ra \mathsf{l}(B_{t})\ra 0.$$
	On the other hand, using the horseshoe lemma to the exact sequence
	$$0\ra X\ra  \mathsf{l}(B)\ra A\ra 0,$$ we obtain the following exact sequence
	$$0\ra X_{t}\ra \mathsf{l}(B)_{t}\ra A_{t}\ra 0,$$
where $X_{t}, \mathsf{l}(B)_{t}$ and $A_{t}$ are corresponding $t$-syzygy objects for $X,\mathsf{l}(B)$ and $A$, respectively.
We mention that both $\mathsf{l}(B_{t})$ and $\mathsf{l}(B)_{t}$ are $t$-syzygy objects of $\mathsf{l}(B)$.
It follows from Schanuel's Lemma that $$\mathsf{l}(B_{t})\oplus P_{0}\cong \mathsf{l}(B)_{t}\oplus Q$$ for some projective objects $P_{0}$ and $Q$ in $\mathscr{A}$.
Thus we have exact sequences
$$0\ra X_{t}\oplus Q\ra \mathsf{l}(B)_{t}\oplus Q\ra A_{t}\ra 0$$
and
$$0\ra \mathsf{l}(K_{0})\ra \mathsf{l}(V_{0})\oplus P\ra \mathsf{l}(B)_{t}\oplus Q\ra 0.$$
	Consider the following pullback diagram:
	$$\xymatrix{&0\ar[d]&0\ar[d]\\
		&\mathsf{l}(K_{0})\ar@{=}[r]\ar[d]&\mathsf{l}(K_{0})\ar[d]\\
		0\ar[r]&Z\ar[r]\ar[d]&\mathsf{l}(V_{0})\oplus P_{0}\ar[d]\ar[r]&A_{t}\ar@{=}[d] \ar[r]&0\\
		0\ar[r] &X_{t}\oplus Q\ar[d]\ar[r]&\mathsf{l}(B)_{t}\oplus Q\ar[d]\ar[r]\ar[d]&A_{t}\ar[r]&0 \\&0&0&}$$
Applying Lemma \ref{IT-lem1} to the first column in this diagram, we have the exact sequence
	$$0\ra X_{t+1}\ra \mathsf{l}(K_{0})\oplus P_{1}\ra Z\ra 0$$ with $P_{1}\in\mathsf{Proj}\mathscr{A}$ and $X_{t+1}$ being a $(t+1)$-syzygy object of $X.$
	Next consider the exact sequence
	\begin{equation}
		\label{se.3.2}
		0\ra K_{1} \ra V_{1}\ra K_{0}\ra 0.
	\end{equation}
	Since $\mathsf{l}$ is an exact functor by assumption, we obtain an exact sequence
	$$0\ra\mathsf{l}(K_{1})\ra \mathsf{l}(V_{1})\oplus P_{1}\ra \mathsf{l}(K_{0})\oplus P_{1}\ra 0.$$
	Now we consider the pullback diagram
	$$\xymatrix{&0\ar[d]&0\ar[d]\\
		&\mathsf{l}(K_{1})\ar@{=}[r]\ar[d]&\mathsf{l}(K_{1})\ar[d]\\
		0\ar[r]&M\ar[r]\ar[d]&\mathsf{l}(V_{1})\oplus P_{1}\ar[d]\ar[r]&Z\ar@{=}[d] \ar[r]&0\\
		0\ar[r] &X_{t+1}\ar[d]\ar[r]&\mathsf{l}(K_{0})\oplus P_{1}\ar[d]\ar[r]\ar[d]&Z\ar[r]&0 \\&0&0.&}$$
	Connecting the second rows in two diagrams together, one gets an exact sequence
	$$0\ra M\ra \mathsf{l}(V_{1})\oplus P_{1}\ra \mathsf{l}(V_{0})\oplus P_{0}\ra A_{t}\ra 0.$$
Continuing in this way, we construct the following exact sequences
	$$0\ra \mathsf{l}(K_{m})\ra Y\ra X_{t+m}\ra 0$$
	and
	$$0\ra Y\ra \mathsf{l}(V_{m})\oplus P_{m}\ra\mathsf{l}(V_{m-1})\oplus P_{m-1}\ra \cdots\ra\mathsf{l}(V_{1})\oplus P_{1}\ra\mathsf{l}(V_{0})\oplus P_{0}\ra A_{t}\ra 0$$
	with each $P_{i}$ projective for $0\leq i\leq m$ and $X_{t+m}$ being a $(t+m)$-syzygy object of X.
	But $K_{m}=0$, thus $Y\cong X_{t+m}.$
	It follows that we obtain the  exact sequence
	$$0\ra X_{t+m}\ra \mathsf{l}(V_{m})\oplus P_{m}\ra\cdots \ra \mathsf{l}(V_{1})\oplus P_{1}\ra \mathsf{l}(V_{0})\oplus P_{0}\ra A_{t}\ra 0,$$
	which completes the proof of the lemma.
\end{proof}

The following theorem gives an estimate of the Igusa-Todorov distance of $\mathscr{A}$ in terms of that of $\mathscr{B}.$
\begin{theorem} \label{thm:IT}
	Let $(\mathscr{B},\mathscr{A},\mathsf{i},\mathsf{e},\mathsf{l})$ be a cleft extension of abelian categories such that $\mathsf{l}$ is exact and $\mathsf{e}$ preserves projectives. If the associated endofunctor $\mathsf{F}$ is nilpotent and $\mathsf{Proj}\mathscr{B}=\add (P)$ for some projective object $P$ in $\mathscr{B}$, then
	$$\IT (\mathscr{B})\leq \IT (\mathscr{A})\leq n(\IT (\mathscr{B})+1)-1,$$
	where $n$ is a positive integer such that $\mathsf{F}^{n}=0$.
\end{theorem}
\begin{proof}
	We first show $\IT (\mathscr{A})\leq n(\IT (\mathscr{B})+1)-1$.
	If $\IT (\mathscr{B})=\infty,$ the inequality holds trivially.
	We assume that $\IT (\mathscr{B})=m$.
	Take an object $X$ in $\mathscr{A}.$
	By Lemma \ref{lem of cleft extension}(2),
	one has the following exact sequences in $\mathscr{A}:$
	\begin{equation}\label{IT.se.1}
		\begin{cases}
			0\ra \mathsf{G}(X)\ra \mathsf{le}(X)\ra X\ra 0;\\
			0\ra \mathsf{G}^{2}(X)\ra \mathsf{l}\mathsf{F}\mathsf{e}(X)\ra \mathsf{G}(X)\ra 0;\\
			\;\;\;\;\;\;\;\;\;\;\;\;\;\;\vdots\\
			$$0\ra \mathsf{G}^{n-1}(X)\ra \mathsf{l}\mathsf{F}^{n-2}\mathsf{e}(X)\ra \mathsf{G}^{n-2}(X)\ra 0;\\
			0\ra \mathsf{G}^{n}(X)\ra \mathsf{l}\mathsf{F}^{n-1}\mathsf{e}(X)\ra \mathsf{G}^{n-1}(X)\ra 0.\\
		\end{cases}
	\end{equation}
Since $\mathsf{F}^{n}=0,$ we have $\mathsf{G}^{n}=0$ by Lemma \ref{lem of cleft extension}(1).
It implies that $\mathsf{G}^{n-1}(X)\cong \mathsf{l}\mathsf{F}^{n-1}\mathsf{e}(X)$.
We mention that $\IT(\mathscr{B})=m$ by hypothesis.
Therefore there exist a nonnegative integer $t$ and an object $V\in \mathscr{B}$ such that the following sequence
$$0\ra V_{m}\ra\cdots\ra V_{1}\ra V_{0}\ra B\ra 0$$
 is exact, where each $V_{i}\in \add (V)$ for every $B\in \Omega^{t}(\mathscr{B})$.
	Applying Lemma \ref{IT-lem2} to exact sequences in \eqref{IT.se.1}, we obtain the following exact sequences
	\begin{equation}
		\label{IT.1}
		0\ra \mathsf{G}(X)_{t+m}\ra \mathsf{l}(V^{1}_{m})\oplus P^{1}_{m}\ra \cdots\ra \mathsf{l}(V^{1}_{1})\oplus P^{1}_{1}\ra \mathsf{l}(V^{1}_{0})\oplus P_{0}^{1} \ra X_{t}\ra 0;
	\end{equation}
	\begin{equation}
		\label{IT.2}
		0\ra \mathsf{G}(X)_{t+2m}\ra \mathsf{l}(V^{2}_{m})\oplus P^{2}_{m}\ra \cdots\ra \mathsf{l}(V^{2}_{1})\oplus P^{2}_{1}\ra \mathsf{l}(V^{2}_{0})\oplus P_{0}^{2} \ra \mathsf{G}(X)_{t+m}\ra 0;
	\end{equation}
	$$\vdots$$
	\begin{multline}
		\label{IT.3}
		0\ra \mathsf{l}\mathsf{F}^{n-1}\mathsf{e}(X)_{t+m(n-1)}\ra \mathsf{l}(V^{n-1}_{m})\oplus P^{n-1}_{m}\ra \cdots\ra \mathsf{l}(V^{n-1}_{1})\oplus P^{n-1}_{1}\\
		\ra \mathsf{l}(V^{n-1}_{0})\oplus P_{0}^{n-1}
		\ra \mathsf{G}^{n-2}(X)_{t+m(n-2)}\ra 0
	\end{multline}
 where each $P^{j}_{i}$ belongs to $\mathsf{Proj}\mathscr{A}$ and each $\mathsf{G}(X)_{t+mj}$ is a $(t+mj)$-syzygy object of $\mathsf{G}(X)$ for $0\leq i\leq m, 1\leq j\leq n-1$, $X_{t}$ is a $t$-syzygy object of $X$ and $\mathsf{l}\mathsf{F}^{n-1}\mathsf{e}(X)_{t+m(n-1)}$ is a $t+m(n-1)$-syzygy object of $\mathsf{l}\mathsf{F}^{n-1}\mathsf{e}(X).$
Let $\mathsf{F}^{n-1}\mathsf{e}(X)_{t+m(n-1)}$ be a $t+m(n-1)$-syzygy object of $\mathsf{F}^{n-1}\mathsf{e}(X).$
Then it can be viewed as a $t$-syzygy object in $\mathscr{B}.$
	Since $\IT (\mathscr{B})=m,$  there is an exact sequence
	\begin{equation}
		\label{se.IT}
		0\ra W_{m}\ra \cdots\ra W_{1}\ra W_{0}\ra \mathsf{F}^{n-1}\mathsf{e}(X)_{t+m(n-1)}\ra 0
	\end{equation}
	with $W_{i}\in \add (V)$ for $0\leq i\leq m$.
	Since $\mathsf{l}$ is exact and preserves projectives by \cite[Lemma 2.2]{GP1}, it follows from Lemma \ref{lem-exact-syzygy} that $\mathsf{l}(\mathsf{F}^{n-1}\mathsf{e}(X)_{t+m(n-1)})$ is a $t+m(n-1)$-syzygy object of $\mathsf{l}\mathsf{F}^{n-1}\mathsf{e}(X).$
By Schanuel's Lemma, we have $$\mathsf{l}\mathsf{F}^{n-1}\mathsf{e}(X)_{t+m(n-1)}\oplus Q_{1}\cong \mathsf{l}(\mathsf{F}^{n-1}\mathsf{e}(X)_{t+m(n-1)})\oplus Q_{2}$$ for some projective objects $Q_{1}$ and $Q_{2}$ in $\mathscr{A}$.
Therefore, applying $\mathsf{l}$ to exact sequence \eqref{se.IT}, we get the following exact sequence
	\begin{equation}
		\label{IT.4}
		0\ra \mathsf{l}(W_{m})\ra\cdots \ra \mathsf{l}(W_{1})\ra\mathsf{l}(W_{0})\oplus Q_{2}\ra \mathsf{l}\mathsf{F}^{n-1}\mathsf{e}(X)_{t+m(n-1)}\oplus Q_{1}\ra 0.
	\end{equation}
	Connecting exact sequences \eqref{IT.1}, \eqref{IT.2}, \eqref{IT.3} and \eqref{IT.4}  together, we have the following exact sequence
	\begin{multline}
\label{IT.5}
		0\ra \mathsf{l}(W_{m})\ra \cdots\ra \mathsf{l}(W_{0})\oplus Q_{2}\ra \mathsf{l}(V^{n-1}_{m})\oplus P^{n-1}_{m}\oplus Q_{1}\ra \cdots\ra \mathsf{l}(V^{n-1}_{0})\oplus P_{0}^{n+1}\longrightarrow \cdots \\
		\longrightarrow \mathsf{l}(V^{1}_{m})\oplus P^{1}_{m}\ra \cdots\ra \mathsf{l}(V^{1}_{0})\oplus P_{0}^{1}\ra X_{t}\ra 0.
	\end{multline}
Since $\mathsf{Proj}\mathscr{B}=\mathrm{add}(P)$ by assumption, it follows from \cite[Lemma 2.3]{Kostas} that each $P_{i}^{j},Q_{1}$ and $Q_{2}$ belong to $\mathrm{add}(\mathsf{l}(P))$ for $0\leq i\leq m, 1\leq j\leq n-1.$	
Hence each term in \eqref{IT.5} belongs to  $\add(\mathsf{l}(V\oplus P))$.
By Definition \ref{def:IT distance}, we get
	$$\IT (\mathscr{A})\leq n(m+1)-1= n(\IT (\mathscr{B})+1)-1.$$
	
On the other hand, suppose that $\IT (\mathscr{A})=k.$
Then there exist a nonnegative integer $s$ and an object $W\in \mathscr{A}$ such that the following sequence
\begin{equation}\label{IT.6}
0\ra W_{k}\ra \cdots \ra W_{1}\ra W_{0}\ra \mathsf{i}(Y)_{s}\ra 0
\end{equation}
is exact for any $Y\in \mathscr{B}$ and  any $s$-syzygy object $\mathsf{i}(Y)_{s}$ of $\mathsf{i}(Y)$,	where each $W_{i}$ belongs to $\add (W)$.
	Since $\mathsf{e}$ is exact and preserves projective objects by assumption, it follows from Lemma \ref{lem-exact-syzygy} that $\mathsf{e}(\mathsf{i}(Y)_{s})$ is a $s$-syzygy object of $\mathsf{ei}(Y)$.
Note that $\mathsf{ei}(Y)\cong Y$.
Then for any $s$-syzygy object $Y_{s}$ of $Y$, by Schanuel's Lemma we have
$\mathsf{e}(\mathsf{i}(Y)_{s})\oplus P_{1}\cong Y_{s}\oplus P_{2}$ with $P_{1}$ and $P_{2}$ projective in $\mathscr{B}$.
Therefore, applying $\sf{e}$ to \eqref{IT.6}, we obtain the following exact sequence
$$0\ra \mathsf{e}(W_{k})\ra \cdots \ra \mathsf{e}(W_{1})\ra \mathsf{e}(W_{0})\oplus P_{1}\ra Y_{s}\oplus P_{2}\ra 0$$
with $\mathsf{e}(W_{k})\in \mathrm{add}(\mathsf{e}(W))$ for each $0\leq i\leq k$.
We denote by $K=\ker(\mathsf{e}(W_{0})\oplus P_{1}\ra Y_{s}\oplus P_{2})$.
Thus we obtain exact sequences
\begin{equation*}
0\ra K\ra \mathsf{e}(W_{0})\oplus P_{1}\ra Y_{s}\oplus P_{2}\ra 0
\end{equation*}
and
\begin{equation}\label{IT.7}
0\ra \mathsf{e}(W_{k})\ra \cdots \ra \mathsf{e}(W_{1})\ra K\ra 0.
\end{equation}
Consider the following pullback diagram:
$$\xymatrix{&&0\ar[d]&0\ar[d]\\
		0\ar[r]&K\ar@{=}[d]\ar[r]&X\ar[d]\ar[r]&Y_{s}\ar[d]\ar[r]&0\\
		0\ar[r]&K\ar[r]&\mathsf{e}(W_{0})\oplus P_{1}\ar[d]\ar[r]&Y_{s}\oplus P_{2}\ar[d] \ar[r]&0\\
		 &&P_{2}\ar[d]\ar@{=}[r]&P_{2}\ar[d]&& \\&&0&0}$$
From the left column,  we obtain an exact sequence
\begin{equation}\label{IT.8}
0\ra K\ra X\ra Y_{s}\ra 0
\end{equation}
  where $X$ is a direct summand of $\mathsf{e}(W_{0})\oplus P_{1}$ as $P_{2}$ is projective.
 Splicing exact sequences \eqref{IT.7} with \eqref{IT.8}, we get the following exact sequence
\begin{equation}\label{IT.9}
 0\ra \mathsf{e}(W_{k})\ra \cdots \ra \mathsf{e}(W_{1})\ra X\ra Y_{s}\ra 0.
 \end{equation}
Since $\mathsf{Proj}\mathscr{B}=\mathrm{add}(P)$ by assumption, and each $\mathsf{e}(W_{i})\in \mathsf{e}(W)$, it follows that each term in \eqref{IT.9} belongs to $\mathrm{add}(\mathsf{e}(W)\oplus P).$ Consequently,
 $\IT (\mathscr{B})\leq \IT(\mathscr{A}).$
\end{proof}


\section{Extension dimensions}

Let $\Lambda$ be a Noetherian algebra. Note that the category of finitely generated right $\Lambda$-modules, denoted by $\modL$, does not have enough injectives.
Therefore, to discuss the extension dimension of a cleft extension of a Noetherian algebra,
we study the relationship between the extension dimension of abelian categories involved in a cleft coextension of abelian categories, rather than a cleft extension of abelian categories. First we recall the definition of extension dimensions from \cite{Be,ZMH}.

Let $\mathcal{U}_1,\mathcal{U}_2,\cdots,\mathcal{U}_n$ be subcategories of $\mathscr{A}$. Define
$$\mathcal{U}_{1}~\bullet ~\mathcal{U}_{2}:={\add}\{A\in \mathscr{A}\mid {\rm there \;exists \;an\; sequence \;}
0\rightarrow U_1\rightarrow  A \rightarrow U_2\rightarrow 0\ {\rm in}\ \mathscr{A}\ {\rm with}\; U_1 \in \mathcal{U}_1 \;{\rm and}\;
U_2 \in \mathcal{U}_2\}.$$
By \cite[Proposition 2.2]{DT}, the operator $\bullet$ is associative, that is,
$(\mathcal{U}_{1}\bullet\mathcal{U}_{2})\bullet\mathcal{U}_{3} =\mathcal{U}_{1}\bullet(\mathcal{U}_{2}\bullet\mathcal{U}_{3}).$
The category $\mathcal{U}_{1}\bullet  \mathcal{U}_{2}\bullet \dots \bullet\mathcal{U}_{n}$
can be inductively described as follows
\begin{align*}
	\mathcal{U}_{1}\bullet  \mathcal{U}_{2}\bullet \dots \bullet\mathcal{U}_{n}:=
	\add \{A\in \mathscr{A}\mid {\rm there \;exists \;an\; sequence}\
	0\rightarrow U\rightarrow  A \rightarrow V\rightarrow 0 \\{\rm in}\ \mathscr{A}\ {\rm with}\; U \in \mathcal{U}_{1} \;{\rm and}\;
	V \in  \mathcal{U}_{2}\bullet \dots \bullet\mathcal{U}_{n}\}.
\end{align*}

For a subclass $\mathcal{U}$ of $\mathscr{A}$, set
$[\mathcal{U}]_{0}=0$, $[\mathcal{U}]_{1}=\add(\mathcal{U})$,
$[\mathcal{U}]_{n}=[\mathcal{U}]_1\bullet [\mathcal{U}]_{n-1}$ for any $n\geq 2$.
If $T$ is an object in $\mathscr{A}$ we write $[ T]_{n}$ instead of $[ \{T \}]_{n}$.

\begin{definition}\label{extension-dimension}
	(\cite{Be})
	The {\bf extension dimension} of $\mathscr{A}$ is defined as
	$$\dim\mathscr{A}:=\inf\{n\geq 0\mid\mathscr{A}= [ T]_{n+1}\ {\rm with}\ T\in\mathscr{A}\},$$
	or $\infty$ if no such an integer exists.
\end{definition}
 Let $\Lambda$ be a Noetherian algebra $\Lambda$. Then the extension dimension of $\modL$ is denoted by $\dim(\Lambda)$.
\begin{remark}
Let $\Lambda$ be an Artin algebra. By \cite[Example 1.6]{Be},  $\Lambda$ is representation-finite if and only if $\dim(\Lambda)=0$.
\end{remark}
\begin{example}{\rm
		Let $\widetilde{\Lambda}_{l,n}=kQ/I$ with $Q$ the quiver
		$$\xymatrix{
			{1}\ar@/^1pc/[r]^{x_{1}}_{\vdots}\ar@/_1pc/[r]_{x_{n}}&2\ar@/^1pc/[r]^{x_{1}}_{\vdots}\ar@/_1pc/[r]_{x_{n}}&3&\cdots
			&l-1\ar@/^1pc/[r]^{x_{1}}_{\vdots}\ar@/_1pc/[r]_{x_{n}}&l
		},$$
and	$$I=(x_{n^{''}}x_{n^{'}}+x_{n^{'}}x_{n^{''}},x^{2}_{n^{'}}\;|1\leq n',n''\leq n).$$
		We know from \cite[Example 2.9]{Zheng-zhang2024} that $\dim (\widetilde{\Lambda}_{l,n})=\min\{l-1,n-1\}.$
		Hence the extension dimension can be arbitrarily large.
	}
\end{example}
We collect the following results which we will need in the proof of Theorem \ref{main thm}.
\begin{lemma}$\mathrm{(}$\cite[Corollary 2.3]{ZMH}$\mathrm{)}$\label{ex-lem}
	For any $T_{1},T_{2}\in \mathscr{A}$ and $m,n\geq 1$, we have
	$[ T_{1}]_{m}\bullet [ T_{2}]_{n}\subseteq [ T_{1}\oplus T_{2}]_{m+n}$.
\end{lemma}


\begin{lemma}{\rm (\cite[Lemma 5.8]{DT})}\label{main lem0}
	If $\mathcal{A}$ has enough projective objects and
	$$0\longrightarrow M\longrightarrow X^{0} \longrightarrow  X^{1}\longrightarrow\cdots\longrightarrow
	X^{n} \longrightarrow 0,$$
	is an exact sequence in $\mathcal{A}$ with $n\geqslant 0$, then
	$$M\in[(X^{n})_{n}]_{1}\bullet[(X^{n-1})_{n-1}]_{1}\bullet\cdots\bullet
	[(X^{1})_{1}]_{1}\bullet[X^{0}]_{1}\subseteq[(\oplus_{i=1}^{n}(X^{i})_{i})\oplus X^{0}]_{n+1},$$
where each $(X^{i})_{i}$ is any $i$-syzygy object of $X^{i}$ for $1\leq i\leq n.$
\end{lemma}

\begin{lemma}{\rm (\cite[Lemma 2.4]{ZMH})}\label{main lem1}
	Let $F:\mathcal{A}\rightarrow \mathcal{B}$ be an exact functor of abelian categories.
	Then $F([T]_{n})\subseteq [F(T)]_{n}$ for any $T\in \mathcal{A}$ and $n\geq 1$.
\end{lemma}

\begin{lemma}\label{ex.lem2}
Let $\mathscr{A}$  be an abelian category with $\mathsf{Proj}\mathscr{A}=\mathrm{add}(P)$ for some projective object $P.$
 Assume that $X$ and $Y$ are in $\mathrm{add}(T)$ for some object $T\in \mathscr{A}$.
 If $X_{t}$ and $Y_{t}$ are $t$-syzygy objects of $X$ and $Y$ respectively for some integer $t$, then there exists a $t$-syzygy object $T_{t}$ of $T$ such that $X_{t}\oplus Y_{t}$ belongs to $\mathrm{add}(T_{t}\oplus P)$.
 \end{lemma}
 \begin{proof}
 Since $X\in \mathrm{add}(T)$ by assumption, there exists an object $Z$ with $X\oplus Z\cong T^{(k)}$ for some integer $k.$
 Then we respectively have $t$-syzygy objects $Z_{t}, (X\oplus Z)_{t}, (T^{(k)})_{t}$ for $Z, X\oplus Z$ and $T^{(k)}$ such that
 $X_{t}\oplus Z_{t}\cong (X\oplus Z)_{t}\cong (T^{(k)})_{t}$.
 Let $T_{t}$ be a $t$-syzygy object of $T$.
 Then $T_{t}^{(k)}$ is also a $t$-syzygy object of $T^{(k)}$.
 By Schanuel's Lemma, one has $(T^{(k)})_{t}\oplus Q'\cong T_{t}^{(k)}\oplus Q$ for some projective objects $Q$ and $Q'.$
 It follows that $X_{t}\oplus Z_{t}\oplus Q'\cong T_{t}^{(k)}\oplus Q.$
 Hence $X_{t}\in \mathrm{add}(T_{t}\oplus P)$ as $\mathsf{Proj}\mathscr{A}=\mathrm{add}(P)$.
 Similarly, we can show that $Y_{t}\in \mathrm{add}(T_{t}\oplus P).$
 Thus we get that $X_{t}\oplus Y_{t}\in \mathrm{add}(T_{t}\oplus P).$
 \end{proof}

\begin{lemma}\label{ex.lem3}
Let $\mathscr{A}$  be an abelian category with $X$ and $Y$ belonging to $\mathscr{A}.$
If $X\in \mathrm{add}(Y)$, then $[X]_{n}\subseteq[Y]_{n}$.
\end{lemma}
\begin{proof}
We proceed by induction on $n$.
Let $n=1$ and let $K\in [X]_{1}=\mathrm{add}(X)$.
Then there exists an object $L$ such that $K\oplus L\cong X^{(k)}$ for some integer $k.$
We mention that $X\in \mathrm{add}(Y)$.
Then one has an object $Z$ such that $X\oplus Z\cong Y^{(t)}$ for some integer $t.$
 Consequently, $X^{(k)}\oplus Z^{(k)}\cong Y^{(kt)}$ and so $K\oplus L\oplus Z^{(k)}\cong Y^{(kt)},$ which implies that $K\in \mathrm{add}(Y)=[Y]_{1}.$
Thus the case for $n=1$ is proved.
By induction hypothesis, we have $[X]_{n-1}\subseteq[Y]_{n-1}.$
Then
\begin{align*}
[X]_{n}&=[X]_{1}\bullet[X]_{n-1}
\subseteq[Y]_{1}\bullet[Y]_{n-1}
=[Y]_{n}.
\end{align*}
The proof is completed.
\end{proof}

\begin{lemma}\label{main lem2}
Let $\mathscr{A}$  be an abelian category with $\mathsf{Proj}\mathscr{A}=\mathrm{add}(P)$ for some projective object $P.$
Assume that $X$ and $Y$ are in $\mathscr{A}$ with $[X]_{1}\subseteq [Y]_{n}$ for $n\geq 1$. Then for any $m$-syzygy object $X_{m}$ of $X$, there exists an $m$-syzygy object $Y_{m}$ of $Y$ such that $ [X_{m}]_{1}\subseteq  [Y_{m}\oplus P]_{n}$ .
\begin{proof}
Let $W\in [X_{m}]_{1}$.
Then there is an object $W'$ with $W\oplus W'\cong X_{m}^{(l)}$ for some integer $l.$
Since $X^{(l)}\in [X]_{1}\subseteq[Y]_{n}$ by assumption, there exist exact sequences
\begin{equation*}
		\begin{cases}
			0\ra Y'_{1}\ra X^{(l)}\oplus Z_{1}\ra Y_{1}\ra 0;\\
			0\ra Y'_{2}\ra Y_{1}\oplus Z_{2}\ra Y_{2}\ra 0;\\
		\;\;\;\;\;\;\;\;\;\;\;\;\;\;\vdots\\
			0\ra Y'_{n-1}\ra Y_{n-2}\oplus Z_{n-1}\ra Y_{n-1}\ra 0,
		\end{cases}
	\end{equation*}
where $Z_{i}\in \mathscr{A},Y'_{i}\in [Y]_{1}$ and $Y_{i}\in[Y]_{n-i}$ for $1\leq i\leq n-1$ (with $Y_{0}=X^{(l)}$).
By horseshoe Lemma, we have
\begin{equation*}
		\begin{cases}
			0\ra (Y'_{1})_{m}\ra (X^{(l)}\oplus Z_{1})_{m}\ra (Y_{1})_{m}\ra 0;\\
			0\ra (Y'_{2})_{m}\ra (Y_{1}\oplus Z_{2})_{m}\ra (Y_{2})_{m}\ra 0;\\
		\;\;\;\;\;\;\;\;\;\;\;\;\;\;\vdots\\
			0\ra (Y'_{n-1})_{m}\ra (Y_{n-2}\oplus Z_{n-1})_{m}\ra (Y_{n-1})_{m}\ra 0,
		\end{cases}
	\end{equation*}
where $(Y'_{i})_{m},(Y_{i})_{m},(Y_{i-1}\oplus Z_{i})_{m}$ are corresponding $m$-syzygy objects for $Y'_{i},Y_{i}$ and $Y_{i-1}\oplus Z_{i}$ respectively.
Then by Schanuel's Lemma, we have the following isomorphism
$$(Y_{i-1}\oplus Z_{i})_{m}\oplus P_{i}\cong (Y_{i-1})_{m}\oplus (Z_{i})_{m}\oplus Q_{i},$$
where $P_{i}$ and $Q_{i}$ are projective,
and each $(Z_{i})_{m}$ is an $m$-syzygy object of $Z_{i}$ for $1\leq i\leq n-1$.
It follows that we have the following exact sequences
\begin{equation}\label{ex.se1}
		\begin{cases}
			0\ra (Y'_{1})_{m}\oplus P_{1}\ra (X^{(l)})_{m}\oplus (Z_{1})_{m}\oplus Q_{1}\ra (Y_{1})_{m}\ra 0;\\
			0\ra (Y'_{2})_{m}\oplus P_{2}\ra (Y_{1})_{m}\oplus (Z_{2})_{m}\oplus Q_{2}\ra (Y_{2})_{m}\ra 0;\\
		\;\;\;\;\;\;\;\;\;\;\;\;\;\;\vdots\\
			0\ra (Y'_{n-1})_{m}\oplus P_{n-1}\ra (Y_{n-2})_{m}\oplus (Z_{n-1})_{m}\oplus Q_{n-1}\ra (Y_{n-1})_{m}\ra 0.
		\end{cases}
	\end{equation}
We mention that $X_{m}^{(l)}\oplus Q\cong (X^{(l)})_{m}\oplus Q'$ for some projective objects $Q'$ and $Q.$
Then from the first exact sequence in \eqref{ex.se1}, we get the following exact sequence
\begin{equation}\label{ex.se2}
0\ra (Y'_{1})_{m}\oplus P_{1}\oplus Q'\ra X_{m}^{(l)}\oplus (Z_{1})_{m}\oplus Q_{1}\oplus Q\ra (Y_{1})_{m}\ra 0.
\end{equation}
Then by \eqref{ex.se1} and \eqref{ex.se2}, we have
\begin{align*}
X_{m}^{(l)}&\in [(Y'_{1})_{m}\oplus P_{1}\oplus Q']_{1}\bullet[(Y_{1})_{m}]_{1}\\
&\subseteq[(Y'_{1})_{m}\oplus P_{1}\oplus Q']_{1}\bullet[(Y'_{2})_{m}\oplus P_{2}]_{1}\bullet [(Y_{2})_{m}]_{1}\\
&\;\;\;\;\;\;\;\;\;\;\;\;\;\;\hdots\\
&\subseteq[(Y'_{1})_{m}\oplus P_{1}\oplus Q']_{1}\bullet[(Y'_{2})_{m}\oplus P_{2}]_{1}\bullet\cdots\bullet [(Y'_{n-1})_{m}\oplus P_{n-1}]_{1}\bullet [(Y_{n-1})_{m}]_{1}\\
&\subseteq[(\oplus_{i=1}^{n}(Y'_{i})_{m})\oplus (Y_{n-1})_{m}\oplus P]_{n} \;\;\;\;\;\;{\rm (by \;Lemma\; \ref{ex-lem} \;and\;}\mathsf{Proj}\mathscr{A}=\mathrm{add}(P)).
\end{align*}
By Lemma \ref{ex.lem2} one has $(\oplus_{i=1}^{n}(Y'_{i})_{m})\oplus (Y_{n-1})_{m}\oplus P\subseteq\mathrm{add}(Y_{m}\oplus P)$ for some $m$-syzygy object $Y_{m}$ of $Y$.
Then it follows from Lemma \ref{ex.lem3} that $[\oplus_{i=1}^{n}(Y'_{i})_{m}\oplus (Y_{n-1})_{m}\oplus P]_{n}\subseteq[Y_{m}\oplus P]_{n}$.
 Consequently, $X_{m}^{(l)}\in [Y_{m}\oplus P]_{n}$ and hence $W\in[Y_{m}\oplus P]_{n}.$ This completes the proof.
\end{proof}
	

	\begin{theorem}\label{main thm}
		Let $(\mathscr{B},\mathscr{A},\mathsf{i},\mathsf{e},\mathsf{r})$ be a cleft coextension of abelian categories with $\mathsf{Proj}\mathscr{A}=\mathrm{add}(P)$ for some projective object $P$ in $\mathscr{A}$. If the induced endofunctor $\mathsf{F}'$ is nilpotent and $\mathsf{r}$ is exact,
		then $$\dim \mathscr{B}\leq \dim \mathscr{A}\leq n(\dim \mathscr{B}+1)-1,$$
		where $n$ is a positive integer such that $\mathsf{F}'^{n}=0$.
	\end{theorem}
\end{lemma}
\begin{proof}
	We first show the right-hand side of the inequality and then the left-hand side of the inequality.
Take an arbitrary  $X\in\mathscr{A}$.
	Then by Lemma \ref{lem of cleft coextension}(2) we have  the following exact sequence in $\mathscr{A}$
	\begin{equation*}
		\begin{cases}
			0\ra X\ra \mathsf{re}(X)\ra \mathsf{G}'(X)\ra 0;\\
			0\ra \mathsf{G}'(X)\ra \mathsf{r}\mathsf{F}'\mathsf{e}(X)\ra \mathsf{G}'^{2}(X)\ra 0;\\
		\;\;\;\;\;\;\;\;\;\;\;\;\;\;\vdots\\
			0\ra \mathsf{G}'^{n-1}(X)\ra \mathsf{r}\mathsf{F}'^{n-1}\mathsf{e}(X)\ra \mathsf{G}'^{n}(X)\ra 0.
		\end{cases}
	\end{equation*}
	Since $\mathsf{F}'^{n}=0,$ we have $\mathsf{G}'^{n}=0$ by Lemma \ref{lem of cleft coextension}(1) and so $\mathsf{G}'^{n-1}(X)\cong \mathsf{r}\mathsf{F}'^{n-1}\mathsf{e}(X)$.
	Thus there is an exact sequence of $\mathscr{A}:$
	$$0\ra X\ra \mathsf{re}(X) \ra \mathsf{r}\mathsf{F}'\mathsf{e}(X)\ra \cdots\ra \mathsf{r}\mathsf{F}'^{n-1}\mathsf{e}(X)\ra  0.$$
	If $\dim \mathscr{B}=\infty$, the inequality holds trivially.
	Assume that $\dim \mathscr{B}=m<\infty$.
	Then $\mathscr{B}=[ T ]_{m+1}$ for some $T\in \mathscr{B}.$
	This implies that each $\mathsf{F}'^{i}\mathsf{e}(X)\in [T]_{m+1}$ for $1\leq i\leq n-1$.
	Since $\sf r$ is exact by assumption, we have
	$\mathsf{r}\mathsf{F}'^{i-1}\mathsf{e}(X)\in[ \mathsf{r}(T)]_{m+1}$ by Lemma \ref{main lem1}; moreover, $[\mathsf{r}\mathsf{F}'^{i-1}\mathsf{e}(X)]_{1}\subseteq [ \mathsf{r}(T)]_{m+1}$.
Then
	\begin{align*}
		X&\in [(\mathsf{r}\mathsf{F}'^{n-1}\mathsf{e}(X))_{n-1}]_{1} \bullet\cdots \bullet [(\mathsf{r}\mathsf{F}'^{}\mathsf{e}(X))_{1}]_{1}\bullet[\mathsf{r}\mathsf{e}(X)]_{1}\;\;\;\;\;\;{\rm (by \;Lemma\; \ref{main lem0})}\\
		&\subseteq [ (\mathsf{r}(T))_{n-1}\oplus P]_{m+1} \bullet\cdots \bullet [(\mathsf{r}(T))_{1}\oplus P]_{m+1}\bullet [ \mathsf{r}(T)\oplus P]_{m+1}\;\;\;\;\;\;{\rm (by \;Lemma\; \ref{main lem2})}\\
		&\subseteq[ \bigoplus_{i=0}^{n-1}(\mathsf{r}(T))_{i}\oplus P]_{n(m+1)},\;\;\;\;\;\;{\rm (by \;Lemma\; \ref{ex-lem})}
	\end{align*}
where each $(\mathsf{r}\mathsf{F}'^{i}\mathsf{e}(X))_{i}$ is an $i$-syzygy object of $\mathsf{r}\mathsf{F}'^{i}\mathsf{e}(X)$ and each $(\mathsf{r}(T))_{i}$ is an $i$-syzygy object of $\sf{r}(T)$ for $0\leq i\leq n-1$.
	Hence $\dim \mathscr{A}\leq n(\dim \mathscr{B}+1)-1$.
	
	On the other hand, we also may assume that $\dim \mathscr{A}=k<\infty.$
	Then $\mathscr{A}=[ S]_{k+1}$ for some $S\in \mathscr{A}.$
	It follows that $\mathsf{i}(B)\in [ S]_{k+1}$ for any $B\in \mathscr{B}.$
	Since $\mathsf{e}$ is exact, it follows from Lemma \ref{main lem1} that
	$$B\cong \mathsf{ei}(B)\in \mathsf{e}[ S]_{k+1}\subseteq[ \mathsf{e}(S)]_{k+1}.$$
	Thus we have $\dim \mathscr{B}\leq\dim \mathscr{A}.$ So we complete this proof.
\end{proof}



\section{Rouquier dimension}
In this section, we concentrate on  the behavior of  dimension of bounded derived categories involved in a cleft extension of module categories.
We begin by recalling the concept of dimension of a triangulated category introduced by Rouquier \cite{Rouquier}.

Let $\mathscr{T}$ be a triangulated category, and let $\mathcal{U}$ and $\mathcal{V}$ be two subcategories of $\mathscr{T}$.
We denote by $\mathcal{U}*\mathcal{V}$ the subcategory consisting of the objects $A$ for which there exists a triangle $$U\rightarrow  A \rightarrow V\rightarrow U[1]$$ with $U\in \mathcal{U}$ and $V\in \mathcal{V}.$
We denote by $\langle \mathcal{U}\rangle_{1}$ the smallest full subcategory of $\mathscr{T}$ which contains $\mathcal{U}$ and is closed under finite direct sums, direct summands and shifts.
Set $\mathcal{U}\diamond \mathcal{V}:=\langle \mathcal{U}* \mathcal{V}\rangle_{1}$ and define inductively
$$\langle \mathcal{U}\rangle_{0}:=0,~\langle \mathcal{U}\rangle_{n+1}=\langle \mathcal{U}\rangle_{1}\diamond\langle\mathcal{U}\rangle_{n}.$$

\begin{definition}(\cite{Rouquier})
	The {\bf dimension} of a triangulated category is defined to be
	$$\der\mathscr{T}:=\inf\{n\geq 0\mid\mathscr{T}= \langle T\rangle_{n+1}\ {\rm with}\ T\in\mathscr{T}\},$$
	or $\infty$ if no such an integer exists.
\end{definition}

\begin{example}(\cite[Example 3.4]{Zheng-huang2022})\label{exa-3.4}
	{\rm Let $\Lambda$ be the Beilinson algebra $kQ/I$ with $Q$ the quiver
		$$\xymatrix{
			&0 \ar@/_1pc/[r]_{x_{n}}\ar@/^1pc/[r]^{x_{0}}_{\vdots}
			&1\ar@/_1pc/[r]_{x_{n}}\ar@/^1pc/[r]^{x_{0}}_{\vdots}
			&2\ar@/_1pc/[r]_{x_{n}}\ar@/^1pc/[r]^{x_{0}}_{\vdots}
			&3&\cdots &n-1\ar@/_1pc/[r]_{x_{n}}\ar@/^1pc/[r]^{x_{0}}_{\vdots}&n
		}$$
		and
		$I=(x_{i}x_{j}-x_{j}x_{i})$, where $0 \leq i, j \leq n$ (see \cite[Example 3.7]{oppermann2010representation}). We know from \cite[Example 3.4]{Zheng-huang2022} or \cite{KrK} that
		$$\der {\mathbf{D}^{b}(\modL)}=n.$$}
\end{example}
The following results are well known.

\begin{theorem} \label{LL-and-gldim}
	Let $\Lambda$ be an Artin algebra. Then the following statements hold:
	\begin{enumerate}
		\item[$(1)$] {\rm (\cite[Proposition 7.37]{Rouquier})} $\der{\mathbf{D}^{b}(\modL)} \leqslant {\rm LL}(\Lambda)-1,$ where ${\rm LL}(\Lambda)$ is the Loewy length of $\Lambda;$
		\item[$(2)$] {\rm (\cite[Proposition 7.4]{Rouquier}, \cite[Proposition 2.6]{KrK})} $\der {\mathbf{D}^{b}(\modL)} \leqslant {\rm gl.dim} \Lambda.$
	\end{enumerate}
\end{theorem}

\begin{definition}\label{def:left perfect}(\cite[Definition 6.4]{KP})
	An endofunctor $\mathsf{F}:\ModB\ra \ModB$ is called {\bf left perfect} if it satisfies the following conditions:
	\begin{enumerate}
		\item $\mathbb{L}_{i}\mathsf{F}^{j}(\mathsf{F}P)=0$ for every projective $B$-module $P$ and all $i,j\geq 1.$
		\item There is a nonnegative integer $n$ such that for every positive integer $p,q\geq1$ with $p+q\geq n+1,$
		we have that $\mathbb{L}_{p}\mathsf{F}^{q}=0.$
	\end{enumerate}
\end{definition}

\begin{definition}(\cite[Definition 4.4]{CL})
	Let $R$ be a ring. We call an $R$-bimodule $M$ {\bf left perfect} provided that it satisfies
	$\mathrm{pd}_{R}M<\infty$ and
	$\mathrm{Tor}_{i}^{R}(M,M^{\oo j})=0$ for all $i,j\geq1.$
\end{definition}

\begin{remark}\label{left perfect}
	If an $R$-bimodule $M$ is left perfect and nilpotent, then the functor $-\oo_{R}M$ is left perfect and nilpotent by \cite[Lemma 6.3]{KP}.
\end{remark}

Let $(\mathscr{B},\mathscr{A},\mathsf{i},\mathsf{e},\mathsf{l})$ be a cleft extension of abelian categories. Then $\sf{i}$ admits a left adjoint by \cite[Lemma 2.2]{GP1}. In what follows, we always denote this left adjoint functor by $\sf{q}.$

The following lemma is contained in the proof of \cite[Proposition 6.3]{Kostas}. For convenience, we provide a proof.
\begin{lemma}\label{lem-derived category}
	Let $(\mathscr{B},\mathscr{A},\mathsf{i},\mathsf{e},\mathsf{l})$ be a cleft extension of abelian categories such that the induced endofunctor $\mathsf{F}$ is left perfect.
	Then we get a diagram of bounded derived categories and triangle functors as below
	$$\scalebox{0.85}{\xymatrixcolsep{3pc}\xymatrix{
			\mathbf{D}^{b}(\mathscr{B})\ar[rr]^{\mathbf{D}^{b}(\sf i)} && \mathbf{D}^{b}(\mathscr{A})\ar[rr]^{\mathbf{D}^{b}(\sf e)} \ar@/_2.0pc/[ll]_{\mathbb{L}\sf q} && \mathbf{D}^{b}(\mathscr{B}) \ar@/_2.0pc/[ll]_{\mathbb{L}\sf l} } }$$
	such that $(\mathbb{L}\sf l,\mathbf{D}^{b}(\sf e))$ and $(\mathbb{L}\sf q,\mathbf{D}^{b}(\sf i))$ are adjoint pairs, and $\mathbf{D}^{b}(\sf e)\mathbf{D}^{b}(\sf i)=\mathrm{Id}_{\mathbf{D}^{b}(\mathscr{B})}.$
\end{lemma}
\begin{proof}
	First we get the following diagram of derived categories
	$$\scalebox{0.85}{\xymatrixcolsep{3pc}\xymatrix{
			\mathbf{D}^{-}(\mathscr{B})\ar[rr]^{\mathbf{D}^{-}(\sf i)} && \mathbf{D}^{-}(\mathscr{A})\ar[rr]^{\mathbf{D}^{-}(\sf e)} \ar@/_2.0pc/[ll]_{\mathbb{L}\sf q}&& \mathbf{D}^{-}(\mathscr{B}) \ar@/_2.0pc/[ll]_{\mathbb{L}\sf l} } }$$
	with $(\mathbb{L}\sf l,\mathbf{D}^{-}(\sf e))$ and $(\mathbb{L}\sf q,\mathbf{D}^{-}(\sf i))$ being adjoint pairs of triangle functors.
	Since $\mathsf{i}$ and $\mathsf{e}$ are exact functors with $\mathsf{e}\mathsf{i}\simeq \mathrm{Id}_{\mathscr{B}},$ it follows that $\mathbf{D}^{-}(\sf e)\mathbf{D}^{-}(\sf i)\simeq \mathrm{Id}_{\mathbf{D}^{-}(\mathscr{B})}.$
	Note that $\mathsf{F}$ is left perfect by assumption. It follows from \cite[Lemma 3.7]{Kostas} and \cite[Proposition 6.6]{KP} that $\mathbb{L}_{n}\mathsf{l}=0=\mathbb{L}_{n}\mathsf{q}$ for $n\gg 0.$
	Then by a similar argument of \cite[Theorem 7.2]{P} we know that $\mathbb{L}\mathsf{l}(Y^{\bullet})\in \mathbf{D}^{b}(\mathscr{A})$ and $\mathbb{L}\mathsf{q}(X^{\bullet})\in \mathbf{D}^{b}(\mathscr{B})$ for any $X^{\bullet}\in \mathbf{D}^{b}(\mathscr{A})$ and $Y^{\bullet}\in \mathbf{D}^{b}(\mathscr{B}).$
	So we have the following diagram
	$$\scalebox{0.85}{\xymatrixcolsep{3pc}\xymatrix{
			\mathbf{D}^{b}(\mathscr{B})\ar[rr]^{\mathbf{D}^{b}(\sf i)} && \mathbf{D}^{b}(\mathscr{A})\ar[rr]^{\mathbf{D}^{b}(\sf e)} \ar@/_2.0pc/[ll]_{\mathbb{L}\sf q} && \mathbf{D}^{b}(\mathscr{B}), \ar@/_2.0pc/[ll]_{\mathbb{L}\sf l} } }$$
	which completes the proof.
\end{proof}

Now we turn our attention to cleft extensions of module categories.
Let $(\ModB,\ModA,\mathsf{i},\mathsf{e},\mathsf{l})$ be a cleft extension of module categories.
Recall from \cite[Subsection 6.4]{KP} that there is an endofunctor $\mathsf{R}:\ModA\ra \ModA$ which appears as the kernel of the unit map of the adjoint pair $(\mathsf{q},\mathsf{i})$:
\begin{equation*}
	0\ra \mathsf{R}\ra \mathrm{Id}_{\ModA}\ra \mathsf{iq}\ra 0.
\end{equation*}
Moreover $\sf R$ is nilpotent whenever $\mathsf{F}$ is nilpotent by \cite[Lemma 6.10]{KP}.

Let $\Lambda$ be a ring. We write $\mathbf{D}^{b}(\modL)$ for the bounded derived category of $\modL$, and let $\mathbf{C}(\modL)$ denote the category of complexes of finitely generated right $\Lambda$-modules.

Given a cleft extension $(\ModB,\ModA,\mathsf{i},\mathsf{e},\mathsf{l})$ of module categories. The following result provides bounds for the dimension of triangulated category $\mathbf{D}^{b}(\modA)$ in terms of the dimension of $\mathbf{D}^{b}(\modB).$

\begin{theorem}\label{thm:dimension of triangulated category}
	Let $A$ and $B$ be two Noetherian algebras, and
	let $(\ModB,\ModA,\mathsf{i},\mathsf{e},\mathsf{l})$ be a cleft extension of module categories.
	If the induced endofunctor $\mathsf{F}$ is left perfect and nilpotent,
	then $$\der \mathbf{D}^{b}(\modB)\leq \der \mathbf{D}^{b}(\modA)\leq n(\der \mathbf{D}^{b}(\modB)+1)-1,$$
	where $n$ is a positive integer such that $\mathsf{F}^{n}=0.$
\end{theorem}
\begin{proof}
	By Lemma \ref{lem-derived category} and Remark \ref{remark:fg}, we have the diagram
	$$\scalebox{0.85}{\xymatrixcolsep{3pc}\xymatrix{
			\mathbf{D}^{b}(\modB)\ar[rr]^{\mathbf{D}^{b}(\sf i)} && \mathbf{D}^{b}(\modA)\ar[rr]^{\mathbf{D}^{b}(\sf e)} \ar@/_2.0pc/[ll]_{\mathbb{L}\sf q} && \mathbf{D}^{b}(\modB) \ar@/_2.0pc/[ll]_{\mathbb{L}\sf l} } }$$
	with $\mathbf{D}^{b}(\mathsf{e})\mathbf{D}^{b}(\mathsf i)=\mathrm{Id}_{\mathbf{D}^{b}(\modB)}.$
	Thus $\mathbf{D}^{b}(\sf e)$ is surjective on objects.
	It follows from \cite[Lemma 3.4]{Rouquier} that
$$\der \mathbf{D}^{b}(\modB)\leq\der \mathbf{D}^{b}(\modA).$$
	
Assume that $\der \mathbf{D}^{b}(\modB)=m$. Then there is a complex $W\in \mathbf{D}^{b}(\modB)$ with $\mathbf{D}^{b}(\modB)=\langle W \rangle_{m+1}.$
	Set $V=\mathsf{i}(W).$
	According to Remark \ref{remark:fg}, we have the following
	cleft extension
	$$\scalebox{0.85}{\xymatrixcolsep{2pc}\xymatrix{
			\modB\ar[rr]^{\sf i} && \modA\ar[rr]^{\sf e} \ar@/_2.0pc/[ll]_{\sf q}  && \modB \ar@/_2.0pc/[ll]_{\sf l} } }$$
	which induces the cleft extension of abelian categories
	$$\scalebox{0.85}{\xymatrixcolsep{2pc}\xymatrix{
			\mathbf{C}(\modB)\ar[rr]^{\sf i} && \mathbf{C}(\modA)\ar[rr]^{\sf e} \ar@/_2.0pc/[ll]_{\sf q}  && \mathbf{C}(\modB). \ar@/_2.0pc/[ll]_{\sf l} } }$$
	Thus for any complex $X^{\bullet}\in \mathbf{D}^{b}(\modA)$ we obtain exact sequences in $\mathbf{C}(\modA)$
	\begin{equation*}
		\begin{cases}
			0\ra \mathsf{R}(X^{\bullet})\ra X^{\bullet}\ra \mathsf{iq}(X^{\bullet})\ra 0;\\
			0\ra \mathsf{R}^{2}(X^{\bullet})\ra \mathsf{R}(X^{\bullet})\ra \mathsf{iq}\mathsf{R}(X^{\bullet})\ra 0;\\
			\;\;\;\;\;\;\;\;\;\;\;\;\;\;\vdots\\
			$$0\ra \mathsf{R}^{n-1}(X^{\bullet})\ra \mathsf{R}^{n-2}(X^{\bullet})\ra \mathsf{iq}\mathsf{R}^{n-2}(X^{\bullet})\ra 0;\\
			0\ra \mathsf{R}^{n}(X^{\bullet})\ra \mathsf{R}^{n-1}(X^{\bullet})\ra \mathsf{iq}\mathsf{R}^{n-1}(X^{\bullet})\ra 0,\\
		\end{cases}
	\end{equation*}
	which in turn give triangles in $\mathbf{D}^{b}(\modA)$
	\begin{equation*}
		\begin{cases}
			\mathsf{R}(X^{\bullet})\ra X^{\bullet}\ra \mathsf{iq}(X^{\bullet})\ra \mathsf{R}(X^{\bullet})[1];\\
			\mathsf{R}^{2}(X^{\bullet})\ra \mathsf{R}(X^{\bullet})\ra \mathsf{iq}\mathsf{R}(X^{\bullet})\ra \mathsf{R}^{2}(X^{\bullet})[1];\\
			\;\;\;\;\;\;\;\;\;\;\;\;\;\;\vdots\\
			\mathsf{R}^{n-1}(X^{\bullet})\ra \mathsf{R}^{n-2}(X^{\bullet})\ra \mathsf{iq}\mathsf{R}^{n-2}(X^{\bullet})\ra \mathsf{R}^{n-1}(X^{\bullet})[1];\\
			\mathsf{R}^{n}(X^{\bullet})\ra \mathsf{R}^{n-1}(X^{\bullet})\ra \mathsf{iq}\mathsf{R}^{n-1}(X^{\bullet})\ra \mathsf{R}^{n}(X^{\bullet})[1].
		\end{cases}
	\end{equation*}
	Since $\mathsf{F}^{n}=0$ for some positive integer $n$ by assumption, it follows from \cite[Lemma 6.10]{KP} and \cite[Proposition 7.4(i)]{Be1} that $\mathsf{R}^{n}=0$ and so
	$\mathsf{R}^{n-1}(X^{\bullet})\cong\mathsf{iq}\mathsf{R}^{n-1}(X^{\bullet})$.
Since $\sf{i}$ is exact, it follows that $$\mathsf{iq}\mathsf{R}^{t}(X^{\bullet})\cong\mathbf{D}^{b}(\mathsf{i})(\mathsf{q}\mathsf{R}^{t}(X^{\bullet}))
	\subseteq \mathbf{D}^{b}(\mathsf{i})(\langle W\rangle_{m+1})\subseteq \langle \mathsf{i}(W)\rangle_{m+1}$$ for $1\leq t\leq n-1$.
	Thus $\mathsf{R}^{n-2}(X^{\bullet})\in \langle \mathsf{i}(W)\rangle_{m+1}\diamond \langle \mathsf{i}(W)\rangle_{m+1}=\langle \mathsf{i}(W)\rangle_{2(m+1)}$ by \cite[Lemma 7.3]{P}.
	Repeating this argument, we infer that $\mathsf{R}(X^{\bullet})\in \langle \mathsf{i}(W)\rangle_{(n-1)(m+1)}$.
	It follows from the triangle
	$$\mathsf{R}(X^{\bullet})\ra X^{\bullet}\ra \mathsf{iq}(X^{\bullet})\ra \mathsf{R}(X^{\bullet})[1]$$
	that $X^{\bullet}\in \langle \mathsf{i}(W)\rangle_{n(m+1)}$.
	Consequently, $\der\mathbf{D}^{b}(\modA)\leq n(\der \mathbf{D}^{b}(\modB)+1)-1.$
\end{proof}




\section{Applications}

In this section, we will apply our main results to Morita context rings, $\theta$-extension rings, tensor rings and arrow removals.
For a Noetherian algebra $\Lambda$, the Igusa-Todorov distance and extension dimension  of $\modL$ are denoted by $\IT(\Lambda)$ and $\dim (\Lambda)$ respectively.

Given an Artin algebra $\Lambda$, if $\IT(\Lambda)\leq 1$, then $\Lambda$ is an Igusa-Todorov algebra introduced by Wei \cite{W}.
It is known that the {\bf finitistic dimension conjecture} holds for every Igusa-Todorov algebra (see \cite[Theorem 1.1]{W}).
Recall that the finitistic dimension $\mathsf{fin. dim}(\Lambda)$ of $\Lambda$ is defined as the supremum of the projective dimension of all finitely generated right $\Lambda$-modules of finite
projective dimension.
The finitistic dimension conjecture asserts that $\mathsf{fin. dim}(\Lambda)<\infty$ for any Artin algebra $\Lambda$. For more on the results of the finitistic dimension conjecture we refer to \cite{GP1,xi2004finitistic,xi2006finitistic,xi2008finitistic,W}.

\begin{corollary}
Let $\Lambda =\left(\begin{smallmatrix}  A & {_{A}}N_{B}\\  {_{B}}M_{A} & B \\\end{smallmatrix}\right)$
	be a Morita context ring which is a Noetherian algebra. Assume that $N\oo_{B}M=0=M\oo_{A}N$.
	\begin{enumerate}
		\item[$(1)$]  If ${N_{B}}, {M_{A}}$, ${_{A}N}$ and ${_{B}M}$ are projective, then  $$\max\{\IT (A),\IT (B)\}\leq\IT (\Lambda)\leq 2\max\{\IT (A),\IT (B)\}+1.$$
Moreover, if $A$ and $B$ are Artin algebras which are syzygy-finite, then $\Lambda$ is an Igusa-Todorov algebra, and thus $\mathsf{fin. dim}(\Lambda)<\infty$.

		\item[$(2)$] If ${N_{B}}$ and ${M_{A}}$ are projective, then $$\max\{\dim (A),\dim (B)\} \leq \dim(\Lambda) \leq 2\max\{\dim (A),\dim (B)\}+1.$$
		
		\item[$(3)$] If $\Tor_{i}^{B}(N,M)=0=\Tor_{i}^{A}(M,N)$ for all $i\geq 1,$ and ${_{A}N}$ and ${_{B}M}$ have finite projective dimension, then
		$$\max\{\der \mathbf{D}^{b}(\modA),\der \mathbf{D}^{b}(\modB)\}\leq\der \mathbf{D}^{b}(\modL)$$$$\leq 2\max\{\der \mathbf{D}^{b}(\modA),\der \mathbf{D}^{b}(\modB)\}+1.$$
	\end{enumerate}
\end{corollary}
\begin{proof}
 We mention that $M\oplus N$ is an $A\times B$-bimodule and $\Lambda\simeq(A\times B)\ltimes (M\oplus N)$; see \cite[Proposition 2.5]{GP}.
Since $\Lambda$ is a Noetherian algebra and $A\times B$-bimodule $M\oplus N$ is finitely generated on both sides, it follows from \cite[Proposition 7.5]{Be1} that $A$ and $B$ are Noetherian algebras.
	By Example \ref{ex:Morita ring} and Remark \ref{remark:fg}, there is a cleft extension of module categories:
	$$\scalebox{0.85}{\xymatrixcolsep{2pc}\xymatrix{
			\modAB\ar[rr]^{\sf i} && \modL\ar[rr]^{\sf e}  && \modAB \ar@/_2.0pc/[ll]_{\sf l}\ar@/^2.0pc/[ll]^{\sf r} } }$$
	Note that $\mathsf{F}(X,Y)=(Y\oo_{B}M,X\oo_{A}N)$.
	It follows from $N\oo_{B}M=0=M\oo_{A}N$ that $\mathsf{F}^{2}=0.$

(1) If $M_{A}, N_{B}, {_{A}N}$ and $_{B}M$ are projective, then it follows from Example \ref{ex:Morita ring} that both $\sf l$ and $\sf r$ are exact. Thus $\sf e$ preserves projectives. Hence statement (1) follows from Theorem \ref{thm:IT}.

(2) If ${N_{B}}$ and ${M_{A}}$ are projective, then we know from Example \ref{ex:Morita ring} that $\sf r$ is exact. Hence statement (2) follows from Theorem \ref{main thm}.

(3) If $\Tor_{i}^{B}(N,M)=0=\Tor_{i}^{A}(M,N)$ for $i\geq 1$ and ${_{A}N}$ and ${_{B}M}$ have finite projective dimension, it follows from \cite[Example 5.4]{Kostas} that $\sf F$ is left perfect. Hence statement (3) follows from Theorem \ref{thm:dimension of triangulated category}.
\end{proof}

\begin{corollary}\label{IT:corollary}
	Let $\Gamma$ be a Noetherian algebra, let $M$ be a $\Gamma$-bimodule and $\theta:M\oo_{\Gamma}M\ra M$ an associative $\Gamma$-bimodule homomorphism with $M^{\oo n}=0$ for some positive integer $n$.
	\begin{enumerate}
		\item[$(1)$] If $_{\Gamma}M$ and $M_{\Gamma}$ are projective, then
		$$\IT (\Gamma)\leq\IT (\Gamma\ltimes_{\theta}M)\leq n(\IT (\Gamma) +1)-1.$$
Moreover, if~$\Gamma$  is a syzygy-finite Artin algebra and $M\otimes_{\Gamma}M=0$, then $\Gamma\ltimes_{\theta}M$ is an Igusa-Todorov algebra, and thus $\mathsf{fin. dim}(\Gamma\ltimes_{\theta}M)<\infty$.
		\item[$(2)$] If $M_{\Gamma}$ is projective, then
        $$\dim (\Gamma) \leq \dim (\Gamma\ltimes_{\theta}M) \leq n(\dim (\Gamma) +1)-1.$$
		
		\item[$(3)$] If $M$ is left perfect, then
		$$\der \mathbf{D}^{b} (\modL)\leq\der \mathbf{D}^{b}(\modG)\leq n(\der \mathbf{D}^{b} (\modL) +1)-1.$$
	\end{enumerate}
	
\end{corollary}
\begin{proof}
	It follows from \cite[Proposition 7.5(iv)]{Be1} that $\Gamma\ltimes_{\theta}M$ is a Noetherian algebra.
	By Example \ref{ex:extension} and Remark \ref{remark:fg}, there is a cleft extension of module categories:
	$$\scalebox{0.85}{\xymatrixcolsep{2pc}\xymatrix{
			\modg\ar[rr]^{\mathbf{Z}} && \modG\ar[rr]^{\mathbf{U}}  && \modg \ar@/^2.0pc/[ll]^{\mathbf{H}=\mathrm{Hom}_{\Gamma}(\Gamma\ltimes_{\theta}M,-)} \ar@/_2.0pc/[ll]_{\mathbf{T}=-\oo_{\Gamma}\Gamma\ltimes_{\theta}M} } }$$
	It is easy to check that the induced endofunctor $\mathsf{F}$ on $\modL$ is naturally isomorphic to $-\oo_{\Gamma}M$
	and  $\mathsf{F}^{n}=0$ as $M^{\oo n}=0$ by assumption.
	Thus endofunctor $\mathsf{F}'$ on $\sMod \Gamma$ is $\Hom_{\Gamma}(M,-)$ and it satisfies $\mathsf{F}'^{n}=0$.

 (1) If $_{\Gamma}M$ and $M_{\Gamma}$ are projective, then it follows that both $\bf T$ and $\bf H$ are exact. Thus $\mathbf{U}$ preserves projectives. Hence statement (1) follows from Theorem \ref{thm:IT}.

(2) If $M_{\Gamma}$ is projective, then it follows that $\bf H$ is exact. Hence statement (2) follows from Theorem \ref{main thm}.

(3) If $M$ is left perfect, then it follows from Remark \ref{left perfect} that $\sf F$ is left perfect. Hence statement (3) follows from Theorem \ref{thm:dimension of triangulated category}.
\end{proof}

Let $\Gamma$ be a ring and $M$ a $\Gamma$-bimodule. We mention that tensor ring $T_{\Gamma}(M)$ can be viewed as a $\theta$-extension of $\Gamma$ by $M'$ under an isomorphism $T_{\Gamma}(M)\cong \Gamma\ltimes_{\theta}M'$ where $M'=M\oplus M^{\oo 2}\oplus\cdots$ and $\theta$ is induced by $M^{\oo k}\oo_{\Gamma}M^{\oo l}\ra M^{\oo k+l}.$ Thus we have the following result by Corollary \ref{IT:corollary}.
\begin{corollary}
	Let $\Gamma$ be a Noetherian algebra and $M$ a $\Gamma$-bimodule with $M^{\oo n}=0$ for some positive integer $n$.
	\begin{enumerate}
		\item[$(1)$] If $_{\Gamma}M$ and $M_{\Gamma}$ are projective, then
		$$\IT (\Gamma)\leq\IT (T_{\Gamma}(M))\leq n(\IT (\Gamma) +1)-1$$
		and
		$$\dim (\Gamma)\leq\dim (T_{\Gamma}(M)) \leq n(\dim (\Gamma) +1)-1.$$
Moreover, if~$\Gamma$  is a syzygy-finite Artin algebra and $M\otimes_{\Gamma}M=0$, then $T_{\Gamma}(M)$ is an Igusa-Todorov algebra, and so $\mathsf{fin. dim}(T_{\Gamma}(M))<\infty$.
		\item[$(2)$] If $M$ is left perfect, then
		$$\der \mathbf{D}^{b} (\modL)\leq\der \mathbf{D}^{b}( {\rm mod}\text{-}T_{\Gamma}(M))\leq n(\der \mathbf{D}^{b}(\modL) +1)-1.$$
	\end{enumerate}
\end{corollary}

\begin{corollary}
	Let $\Gamma$ be a Noetherian algebra and $M$ a $\Gamma$-bimodule with $M^{\oo n}=0$ for some positive integer $n$.
	\begin{enumerate}
		\item[$(1)$] If $_{\Gamma}M$ and $M_{\Gamma}$ are projective, then
		$$\IT (\Gamma)\leq\IT (\Gamma\ltimes M)\leq n(\IT (\Gamma) +1).$$
Moreover, if~$\Gamma$  is a syzygy-finite Artin algebra and $M\otimes_{\Gamma}M=0$, then $\Gamma\ltimes M$ is an Igusa-Todorov algebra, and so $\mathsf{fin. dim}(\Gamma\ltimes M)<\infty$.
		
		\item[$(2)$] If $M_{\Gamma}$ is projective, then $$\dim (\Gamma)\leq\dim (\Gamma\ltimes M) \leq n(\dim (\Gamma) +1)-1.$$
		
		\item[$(3)$] If $M$ is left perfect, then
		$$\der \mathbf{D}^{b} (\modL)\leq\der \mathbf{D}^{b}({\rm mod}\text{-}\Gamma\ltimes M)\leq n(\der \mathbf{D}^{b} (\modL) +1)-1.$$
	\end{enumerate}
\end{corollary}

\begin{proof}
	Since $\Gamma\ltimes M\cong \Gamma\ltimes_{\theta}M$ with $\theta=0$, the result follows from Corollary \ref{IT:corollary}.
\end{proof}

Recall from \cite{GP1} that the arrow removal operation on an admissible path algebra induces a cleft extension with certain homological properties.
For more details, we refer to \cite{GP1,EPS}.

\begin{theorem}\label{arrow removal}
	Let $\Lambda=kQ/I$ be an admissible quotient of a path algebra $kQ$ over a field $k.$
	Suppose that there are arrows $a_{i}:\upsilon_{e_{i}}\ra \upsilon_{f_{i}}$ in $Q$ for $i=1,2,\cdots, t$ which do not occur in a set of minimal generators of $I$ in $kQ$ and $\Hom_{\Lambda}(e_{i}\Lambda,f_{j}\Lambda)=0$ for all $i,j$ in $\{1,2,\cdots,t\}$.
	Let $\Gamma=\Lambda/\Lambda\{\overline{a_{i}}\}_{i=1}^{t}\Lambda$.
	Then the following hold:
	\begin{enumerate}
		\item[$(1)$] $\IT (\Gamma) \leq \IT(\Lambda) \leq 2\IT(\Gamma)+1.$
		\item[$(2)$] $\dim (\Gamma) \leq \dim(\Lambda) \leq 2\dim(\Gamma)+1.$
		\item[$(3)$] $\der \mathbf{D}^{b}(\modg)\leq\der\mathbf{D}^{b}(\modL)\leq 2\der \mathbf{D}^{b}(\modg)+1.$
	\end{enumerate}
	\begin{proof}
		By \cite[Proposition 4.6]{GP1}, there is a cleft extension
		$$\scalebox{0.85}{\xymatrixcolsep{3pc}\xymatrix{
				\modg\ar[rr]^{\mathsf{i}=\mathrm{Hom}_{\Gamma}(_{\Lambda}\Gamma_{\Gamma},-)} && \modL\ar[rr]^{\mathsf{e}=\mathrm{Hom}_{\Lambda}(_{\Gamma}\Lambda_{\Lambda},-)} \ar@/_2.0pc/[ll]_{\mathsf{q}=-\otimes_{\Lambda}\Gamma_{\Gamma}} \ar@/^2.0pc/[ll]^{\mathsf{p}=\mathrm{Hom}_{\Lambda}(_{\Gamma}\Gamma_{\Lambda},-)} && \modg. \ar@/_2.0pc/[ll]_{\mathsf{l}=-\otimes_{\Gamma}\Lambda_{\Lambda}} \ar@/^2.0pc/[ll]^{\mathsf{r}=\mathrm{Hom}_{\Gamma}(_{\Lambda}\Lambda_{\Gamma},-)} } }$$
		such that $\sf l$ and $\sf r$ are exact, $\sf e$ preserves projectives and $\mathsf{F}^{2}=0$, where $\mathsf{F}$ is the induced endofunctor of the above cleft extension.
		Since $\mathsf {el}\simeq \mathsf{F}\oplus \mathrm{Id}_{\modB}$ by \cite[Lemma 2.3]{GP1}, it follows that $\mathsf{F}$ is exact and preserves projectives. Hence $\mathsf{F}$ is left perfect in a trivial way.
		So this corollary follows from Theorems \ref{thm:IT}, \ref{main thm} and \ref{thm:dimension of triangulated category}.
	\end{proof}
\end{theorem}

\section{Examples}

\def\Up{\mathrm{U}}
\def\Down{\mathrm{D}}
\def\e{\varepsilon}
\def\modcat{\mathrm{mod}\text{-}}
\def\I{\mathrm{I}}
\def\II{\mathrm{II}}
\def\rad{\mathrm{rad}}
\def\soc{\mathrm{soc}}
\def\lineardim{\mathrm{dim}_{k}}

In this section, we provide two examples illustating Theorem \ref{arrow removal}. They can be seen as examples for Theorems \ref{thm:IT}, \ref{main thm} and \ref{thm:dimension of triangulated category}.

\begin{example}
Let $A = kQ/I$ be a finite dimensional algebra whose bound quiver $(Q,I)$ is given by Figure \ref{fig}.
 This is a gentle algebra. Gentle algebras are introduced by Assem and Skowro\'{n}ski in \cite{AS1987};
their module categories and derived categories have been studied by using
quiver methods \cite[etc]{BD2017, LZ2019, Z2016, Z2019, ZH2016} and
geometric models \cite[etc]{BCS2021, LZ2021, OPS2018, QZZ2022}.
\begin{figure}[htbp]
  \centering
\begin{tikzpicture}[scale=2]
\draw[rotate around={ 33:(0,0)}][line width=3pt][cyan!50][dashed] (0,0.75) arc(90:150:0.75);
\draw[rotate around={213:(0,0)}][line width=3pt][cyan!50][dashed] (0,0.75) arc(90:150:0.75);
\draw[rotate around={  0:(0,0)}][line width=3pt][cyan!50][dashed] (-1.3,0) arc(0:82:0.7);
\draw[rotate around={180:(0,0)}][line width=3pt][cyan!50][dashed] (-1.3,0) arc(0:82:0.7);
%
\foreach \x in {0,60,...,300}
\draw[->][rotate around = { 0- \x:(0,0)}] (-2,0) node[blue]{$\bullet$};
\foreach \x in {60,120,240,300}
\draw[->][rotate around = { 0- \x:(0,0)}] (-2,0) node[red]{$\bullet$};
\foreach \x in {0,60,120}
\draw[->][rotate around = {-5- \x:(0,0)}] (-2,0) arc ( 180: 130:2);
\foreach \x in {0,60,120}
\draw[->][rotate around = { 5+ \x:(0,0)}] (-2,0) arc (-180:-130:2);
\draw (-0.20, 0.00) node[blue]{$\bullet$};
\draw (-0.85, 1.50) node[ red]{$\bullet$};
\draw (-0.85,-1.50) node[ red]{$\bullet$};
\draw[->] (-0.40, 0.00) -- (-1.80, 0.00);
\draw[->] (-0.80, 1.40) -- (-0.20, 0.20);
\draw[<-] (-0.80,-1.40) -- (-0.20,-0.20);
\draw[->] (-0.73, 1.53) to[out=19,in=161] ( 0.73, 1.53);
\draw[rotate around = {180:(0,0)}] (-0.20, 0.00) node[blue]{$\bullet$};
\draw[rotate around = {180:(0,0)}] (-0.85, 1.50) node[ red]{$\bullet$};
\draw[rotate around = {180:(0,0)}] (-0.85,-1.50) node[ red]{$\bullet$};
\draw[<-][rotate around = {180:(0,0)}] (-0.40, 0.00) -- (-1.80, 0.00);
\draw[<-][rotate around = {180:(0,0)}] (-0.80, 1.40) -- (-0.20, 0.20);
\draw[->][rotate around = {180:(0,0)}] (-0.80,-1.40) -- (-0.20,-0.20);
\draw[<-][rotate around = {180:(0,0)}] (-0.73, 1.53) to[out=19,in=161] ( 0.73, 1.53);
\draw[->][line width=1pt][rotate around = {-5- 60:(0,0)}][red] (-2,0) arc ( 180: 130:2);
\draw[->][line width=1pt][rotate around = { 5+ 60:(0,0)}][red] (-2,0) arc (-180:-130:2);
\draw[->][line width=1pt][red] (-0.73, 1.53) to[out=19,in=161] ( 0.73, 1.53);
\draw[<-][line width=1pt][red][rotate around = {180:(0,0)}] (-0.73, 1.53) to[out=19,in=161] ( 0.73, 1.53);
\draw[rotate around = {   0:(0,0)}] (-2.3,0) node[blue]{$1$};
\draw[rotate around = { -60:(0,0)}] (-2.3,0) node[ red]{$2_{\Up}$};
\draw[rotate around = {-120:(0,0)}] (-2.3,0) node[ red]{$3_{\Up}$};
\draw[rotate around = {-180:(0,0)}] (-2.3,0) node[blue]{$4$};
\draw[rotate around = {-240:(0,0)}] (-2.3,0) node[ red]{$3_{\Down}$};
\draw[rotate around = {-300:(0,0)}] (-2.3,0) node[ red]{$2_{\Down}$};
\draw (-1.00, 1.25) node[ red]{$5_{\Up}$};
\draw ( 1.00, 1.25) node[ red]{$6_{\Up}$};
\draw (-1.00,-1.25) node[ red]{$5_{\Down}$};
\draw ( 1.00,-1.25) node[ red]{$6_{\Down}$};
\draw (-0.45, 0.23) node[blue]{$7$};
\draw ( 0.45,-0.23) node[blue]{$8$};
\draw[rotate around = {  60:(0,0)}] (0,2.2) node{$a_{\Up}$};
\draw[rotate around = {   0:(0,0)}] (0,2.2) node[red]{$b_{\Up}$};
\draw[rotate around = { -60:(0,0)}] (0,2.2) node{$c_{\Up}$};
\draw[rotate around = { 120:(0,0)}] (0,2.2) node{$a_{\Down}$};
\draw[rotate around = { 180:(0,0)}] (0,2.2) node[red]{$b_{\Down}$};
\draw[rotate around = { 240:(0,0)}] (0,2.2) node{$c_{\Down}$};
\draw( 0.00, 1.45) node[red]{$d_{\Up}$};
\draw( 0.00,-1.45) node[red]{$d_{\Up}$};
\draw(-0.65, 0.80) node{$\alpha_{\Up}$};
\draw(-0.65,-0.80) node{$\alpha_{\Down}$};
\draw(-1.00, 0.20) node{$\beta$};
\draw( 0.65, 0.80) node{$\gamma_{\Up}$};
\draw( 0.65,-0.80) node{$\gamma_{\Down}$};
\draw( 1.00, 0.20) node{$\delta$};
\end{tikzpicture}
  \caption{Bound quiver $(Q,I)$ with $I = \langle \alpha_{\Up}\beta, \beta a_{\Up}, c_{\Down}\delta, \delta \gamma_{\Down}  \rangle$}
  \label{fig}
\end{figure}
In this instance, take $e_1=2_{\Up}$, $e_2=5_{\Up}$, $e_3=2_{\Down}$, $e_4=5_{\Down}$,
$f_1=3_{\Up}$, $f_2=6_{\Up}$, $f_3=6_{\Down}$ and $f_4=6_{\Down}$.
Then the arrows $a_1:=b_{\Up}$, $a_2:=d_{\Up}$, $a_3:=b_{\Down}$, $a_4:=d_{\Down}$ do not occur in a set of minimal generators
of $I$, and one can check that
\[ \lineardim \Hom_A(\e_{e_i}A, \e_{f_j}A) = \lineardim (\e_{f_j}A\e_{e_i})
 = \text{the number of paths from $f_j$ to $e_i$} = 0 \]
for all $i,j\in\{1,2,3,4\}$, where $\e_{e_{i}}:=e_{i}+I.$
It follows that
\[ \Hom_A(\e_{e_i}A, \e_{f_j}A) = 0, \ \forall i,j\in\{1,2,3,4\}. \]
Then we have a cleft extension
$$
\scalebox{0.85}{\xymatrixcolsep{3pc}\xymatrix{
   \modcat\Gamma
   \ar[rr]^{\mathrm{Hom}_{\Gamma}(_{A}\Gamma_{\Gamma},-)}
&& \modcat A
   \ar[rr]^{\mathrm{Hom}_{A}(_{\Gamma}A_{A},-)}
   \ar@/_2.0pc/[ll]_{-\otimes_{A}\Gamma_{\Gamma}}
   \ar@/^2.0pc/[ll]^{\mathrm{Hom}_{A}(_{\Gamma}\Gamma_{A},-)}
&& \modcat\Gamma,
   \ar@/_2.0pc/[ll]_{-\otimes_{\Gamma}A_{A}}
   \ar@/^2.0pc/[ll]^{\mathrm{Hom}_{\Gamma}(_{A}A_{\Gamma},-)} } }$$
where $\Gamma = A/A\{b_{\Up}, d_{\Up}, b_{\Down}, d_{\Down}\}A$ is isomorphic to $kQ'/I'$,
and the bound quiver $(Q',I')$ of it is shown in Figure \ref{fig2}.
Here, $Q'$ is not connected and $I'=I$. Quiver $Q'$ has two connected components $Q'_{\I}$ and $Q'_{\II}$,
then $\Gamma = \Gamma_1 \times \Gamma_2$, where
$\Gamma_1 = kQ'_{\I}/\langle \alpha_{\Up}\beta, \beta a_{\Up} \rangle$ and
$\Gamma_2 = kQ'_{\II}/\langle  c_{\Down}\delta, \delta \gamma_{\Down}  \rangle$.
It follows that $\modcat \Gamma_1 \times \modcat \Gamma_2$.
%
\begin{figure}[htbp]
  \centering
\begin{tikzpicture}[scale=2]
\draw[rotate around={ 33:(0,0)}][line width=3pt][cyan!50][dashed] (0,0.75) arc(90:150:0.75);
\draw[rotate around={213:(0,0)}][line width=3pt][cyan!50][dashed] (0,0.75) arc(90:150:0.75);
\draw[rotate around={  0:(0,0)}][line width=3pt][cyan!50][dashed] (-1.3,0) arc(0:82:0.7);
\draw[rotate around={180:(0,0)}][line width=3pt][cyan!50][dashed] (-1.3,0) arc(0:82:0.7);
%
\foreach \x in {0,120,180,300}
\draw[->][rotate around = { 0- \x:(0,0)}] (-2,0) node[blue]{$\bullet$};
\foreach \x in {60,120,240,300}
\draw[->][rotate around = { 0- \x:(0,0)}] (-2,0) node[red]{$\bullet$};
\foreach \x in {0, 120}
\draw[->][rotate around = {-5- \x:(0,0)}] (-2,0) arc ( 180: 130:2);
\foreach \x in {0, 120}
\draw[->][rotate around = { 5+ \x:(0,0)}] (-2,0) arc (-180:-130:2);
\draw (-0.20, 0.00) node[blue]{$\bullet$};
\draw (-0.85, 1.50) node[ red]{$\bullet$};
\draw (-0.85,-1.50) node[ red]{$\bullet$};
\draw[->] (-0.40, 0.00) -- (-1.80, 0.00);
\draw[->] (-0.80, 1.40) -- (-0.20, 0.20);
\draw[<-] (-0.80,-1.40) -- (-0.20,-0.20);
\draw[rotate around = {180:(0,0)}] (-0.20, 0.00) node[blue]{$\bullet$};
\draw[rotate around = {180:(0,0)}] (-0.85, 1.50) node[ red]{$\bullet$};
\draw[rotate around = {180:(0,0)}] (-0.85,-1.50) node[ red]{$\bullet$};
\draw[<-][rotate around = {180:(0,0)}] (-0.40, 0.00) -- (-1.80, 0.00);
\draw[<-][rotate around = {180:(0,0)}] (-0.80, 1.40) -- (-0.20, 0.20);
\draw[->][rotate around = {180:(0,0)}] (-0.80,-1.40) -- (-0.20,-0.20);
%
\draw[rotate around = {   0:(0,0)}] (-2.3,0) node[blue]{$1$};
\draw[rotate around = { -60:(0,0)}] (-2.3,0) node[ red]{$2_{\Up}$};
\draw[rotate around = {-120:(0,0)}] (-2.3,0) node[ red]{$3_{\Up}$};
\draw[rotate around = {-180:(0,0)}] (-2.3,0) node[blue]{$4$};
\draw[rotate around = {-240:(0,0)}] (-2.3,0) node[ red]{$3_{\Down}$};
\draw[rotate around = {-300:(0,0)}] (-2.3,0) node[ red]{$2_{\Down}$};
\draw (-1.00, 1.25) node[ red]{$5_{\Up}$};
\draw ( 1.00, 1.25) node[ red]{$6_{\Up}$};
\draw (-1.00,-1.25) node[ red]{$5_{\Down}$};
\draw ( 1.00,-1.25) node[ red]{$6_{\Down}$};
\draw (-0.45, 0.23) node[blue]{$7$};
\draw ( 0.45,-0.23) node[blue]{$8$};
\draw[rotate around = {  60:(0,0)}] (0,2.2) node{$a_{\Up}$};
\draw[rotate around = { -60:(0,0)}] (0,2.2) node{$c_{\Up}$};
\draw[rotate around = { 120:(0,0)}] (0,2.2) node{$a_{\Down}$};
\draw[rotate around = { 240:(0,0)}] (0,2.2) node{$c_{\Down}$};
\draw(-0.65, 0.80) node{$\alpha_{\Up}$};
\draw(-0.65,-0.80) node{$\alpha_{\Down}$};
\draw(-1.00, 0.20) node{$\beta$};
\draw( 0.65, 0.80) node{$\gamma_{\Up}$};
\draw( 0.65,-0.80) node{$\gamma_{\Down}$};
\draw( 1.00, 0.20) node{$\delta$};
\draw (-2.6,1) node{$Q'_{\I}$} (2.6,-1) node{$Q'_{\II}$};
\end{tikzpicture}
  \caption{$\Gamma = A/A\{b_{\Up}, d_{\Up}, b_{\Down}, d_{\Down}\}A$}
  \label{fig2}
\end{figure}
Then we have:
\begin{itemize}
\item[(1)] $A$ is representation-infinite.

Recall that for each bound quiver of a gentle algebra,
we can define the {\bf formal inverse} of any arrow $a$ and write it as $a^{-1}$
and, naturally, define the sink/source of $a^{-1}$ is the source/sink of $a$.
A {\bf string} $s$ is a sequence $(s_1,s_2,\cdots, s_l)$ such that:
\begin{itemize}
  \item each $s_i$ ($1\leq i\leq l$) is either an arrow or a formal inverse of an arrow;
  \item if $s_{i}$ and $s_{i+1}$ are arrows, then $s_{i}s_{i+1}\ne 0$;
  \item if $s_{i}$ and $s_{i+1}$ are arrows' formal inverse, then $s_{i+1}^{-1}s_{i}^{-1}\ne 0$.
\end{itemize}
Furthermore, we say that a string $s$ is a {\bf band} if it satisfies the following conditions:
\begin{itemize}
  \item the sink of $s_l$ and the source of $s_1$ coincide;
  \item $s$ is not a non-trivial power of any string;
  \item $s^2$ is a string.
\end{itemize}
Then all indecomposable modules over a gentle algebra can be described by strings and bands;
see the works of Butler and in Ringel \cite[Section 3]{BR1987}.
For example, $s = a_{\Up}b_{\Up}c_{\Up}c_{\Down}^{-1}b_{\Down}^{-1}a_{\Down}^{-1}$ is a band,
it describes a family of indecomposable modules $B(s,n,\lambda)$ ($\lambda\ne 0$) whose quiver representations are of the form shown in Figure \ref{fig:band}.
One can check that
\begin{center}
  $B(s,n,\lambda) \cong B(s,n,\lambda')$ if and only if $\lambda = \lambda'$,
\end{center}
then $A$ is representation-infinite since all algebraically closed fields are always infinite fields.
\begin{figure}[htbp]
  \centering
\begin{tikzpicture}[scale=1.35]
\draw[rotate around={ 33:(0,0)}][line width=3pt][cyan!50][dashed] (0,0.75) arc(90:150:0.75);
\draw[rotate around={213:(0,0)}][line width=3pt][cyan!50][dashed] (0,0.75) arc(90:150:0.75);
\draw[rotate around={  0:(0,0)}][line width=3pt][cyan!50][dashed] (-1.3,0) arc(0:82:0.7);
\draw[rotate around={180:(0,0)}][line width=3pt][cyan!50][dashed] (-1.3,0) arc(0:82:0.7);
%
\foreach \x in {0,180}
\draw[rotate around = { 0- \x:(0,0)}] (-2,0) node[blue]{$k^{\oplus n}$};
\foreach \x in {60,120,240,300}
\draw[rotate around = { 0- \x:(0,0)}] (-2,0) node[red]{$k^{\oplus n}$};
\foreach \x in {0,60,120}
\draw[->][rotate around = {-5- \x:(0,0)}] (-2,0) arc ( 180: 130:2);
\foreach \x in {0,60,120}
\draw[->][rotate around = { 5+ \x:(0,0)}] (-2,0) arc (-180:-130:2);
\draw (-0.20, 0.00) node[blue]{$0$};
\draw (-0.85, 1.50) node[ red]{$0$};
\draw (-0.85,-1.50) node[ red]{$0$};
\draw[->] (-0.40, 0.00) -- (-1.80, 0.00);
\draw[->] (-0.80, 1.40) -- (-0.20, 0.20);
\draw[<-] (-0.80,-1.40) -- (-0.20,-0.20);
\draw[->] (-0.73, 1.53) to[out=19,in=161] ( 0.73, 1.53);
\draw[rotate around = {180:(0,0)}] (-0.20, 0.00) node[blue]{$0$};
\draw[rotate around = {180:(0,0)}] (-0.85, 1.50) node[ red]{$0$};
\draw[rotate around = {180:(0,0)}] (-0.85,-1.50) node[ red]{$0$};
\draw[<-][rotate around = {180:(0,0)}] (-0.40, 0.00) -- (-1.80, 0.00);
\draw[<-][rotate around = {180:(0,0)}] (-0.80, 1.40) -- (-0.20, 0.20);
\draw[->][rotate around = {180:(0,0)}] (-0.80,-1.40) -- (-0.20,-0.20);
\draw[<-][rotate around = {180:(0,0)}] (-0.73, 1.53) to[out=19,in=161] ( 0.73, 1.53);
\draw[->][line width=1pt][rotate around = {-5- 60:(0,0)}][red] (-2,0) arc ( 180: 130:2);
\draw[->][line width=1pt][rotate around = { 5+ 60:(0,0)}][red] (-2,0) arc (-180:-130:2);
\draw[->][line width=1pt][red] (-0.73, 1.53) to[out=19,in=161] ( 0.73, 1.53);
\draw[<-][line width=1pt][red][rotate around = {180:(0,0)}] (-0.73, 1.53) to[out=19,in=161] ( 0.73, 1.53);

\draw[rotate around = {  60:(0,0)}] (0,2.2) node{$\pmb{J}_n(\lambda)$};
\draw[rotate around = {   0:(0,0)}] (0,2.2) node[red]{$\pmb{E}$};
\draw[rotate around = { -60:(0,0)}] (0,2.2) node{$\pmb{E}$};
\draw[rotate around = { 120:(0,0)}] (0,2.2) node{$\pmb{E}$};
\draw[rotate around = { 180:(0,0)}] (0,2.2) node[red]{$\pmb{E}$};
\draw[rotate around = { 240:(0,0)}] (0,2.2) node{$\pmb{E}$};
\draw( 0.00, 1.45) node[red]{$0$};
\draw( 0.00,-1.45) node[red]{$0$};
\draw(-0.65, 0.80) node{$0$};
\draw(-0.65,-0.80) node{$0$};
\draw(-1.00, 0.20) node{$0$};
\draw( 0.65, 0.80) node{$0$};
\draw( 0.65,-0.80) node{$0$};
\draw( 1.00, 0.20) node{$0$};
\draw(-2.5,0) node[left]{$\pmb{J}_n(\lambda) =
\left(\begin{smallmatrix}
\lambda & 1 & &  \\
& \lambda & \ddots &  \\
& & \ddots& 1 \\
& & & \lambda  \\
\end{smallmatrix}
\right)_{n\times n}$};
\draw( 2.5,0) node[right]{$\pmb{E} =
\left(\begin{smallmatrix}
1 & & &  \\
& 1 & &  \\
& & \ddots& \\
& & & 1  \\
\end{smallmatrix}
\right)_{n\times n}$};
\end{tikzpicture}
  \caption{Band modules $B(s,n,\lambda)$}
  \label{fig:band}
\end{figure}

\item[(2)] $\Gamma$ is representation-finite.

The radicals of indecomposable projective $\Gamma_1$-modules $P(7)_{\Gamma_1}$ and $P(1)_{\Gamma_1}$ are decomposable. Precisely we have
\[ \rad P(7)_{\Gamma_1}
= \rad
  \left(
    \begin{smallmatrix}
     & 7 & \\ 1 & & 5_{\Down} \\ 2_{\Down} & &
    \end{smallmatrix}
  \right)_{\Gamma_1}
\cong \left({^{~1}_{2_{\Down}}}\right)_{\Gamma_1} \oplus (5_{\Down})_{\Gamma_1};
\]
\[
\rad P(1)_{\Gamma_1}
= \rad   \left(
    \begin{smallmatrix}
     & 1 & \\ 2_{\Up} & & 2_{\Down}
    \end{smallmatrix}
  \right)_{\Gamma_1}
\cong (2_{\Up})_{\Gamma_1} \oplus (2_{\Down})_{\Gamma_2}.
\]
Note that $\Gamma_1$ is a gentle algebra whose dimension (as $k$-vector space) is $13$,
and for each vertex $v$ of $Q_{\I}'$, there is at most one arrow with sink $v$,
then, by using \cite[Theorem 1.1 (1)]{WLL2024}, the number of indecomposable $\Gamma_1$-module is
\[ \lineardim \Gamma_1
 + \lineardim \left({^{~1}_{2_{\Down}}}\right)_{\Gamma_1} \cdot \lineardim (5_{\Down})_{\Gamma_1}
 + \lineardim (2_{\Up})_{\Gamma_1} \cdot (2_{\Down})_{\Gamma_2}
= 13 + 2 + 1 = 16.
\]
Dually, since $\Gamma_2$ is a gentle algebra with dimension $\lineardim \Gamma_2 = 13,$
and the quotient
\[ E(4)_{\Gamma_2}/\soc(E(4)_{\Gamma_2}) \cong (3_{\Up})_{\Gamma_2}\oplus (3_{\Down})_{\Gamma_2}\] and
\[ E(8)_{\Gamma_2}/\soc(E(4)_{\Gamma_2}) \cong \left({^{3_{\Up}}_{~4}}\right)_{\Gamma_2} \oplus (6_{\Up})_{\Gamma_2} \]
 are decomposable. By \cite[Theorem 1.1 (2)]{WLL2024}, the number of indecomposable $\Gamma_2$-modules is
\[ \lineardim \Gamma_2
 + \lineardim (3_{\Up})_{\Gamma_2} \cdot \lineardim (3_{\Down})_{\Gamma_2}
 + \lineardim \left({^{3_{\Up}}_{~4}}\right)_{\Gamma_2} \cdot (6_{\Up})_{\Gamma_2}
= 13 + 1 + 2 = 16. \]
Thus the number of indecomposable $\Gamma$-modules is $32$, and then $\Gamma$ is representation-finite.
Therefore, by Definitions \ref{def:IT distance} and \ref{extension-dimension},
we can see that $\IT (\Gamma) = 0 = \dim (\Gamma)$.

\item[(3)] $A$ is derived-infinite.

An algebra $\Lambda$ is called {\bf derived-finite} if the number of indecomposable objects in its bounded derived category $\mathbf{D}^b(\Lambda)$ is finite (up to isomorphisms and shifts).
Then an algebra $\Lambda$ is called {\bf derived-infinite} if it does not derived-finite.

Note that the bound quiver $(Q,I)$ of $A$ has a hereditary subquiver of Euclid type $\widetilde{\mathbb{A}}$
which is naturally induced by the band $s = a_{\Up}b_{\Up}c_{\Up}c_{\Down}^{-1}b_{\Down}^{-1}a_{\Down}^{-1}$,
and the number of clockwise arrows equals to that of anti-clockwise arrows,
then $A$ is derived-infinite (indeed, it is strongly derived-unbounded,
the definition of strongly derived-unbounded can be found in reference \cite[Definition 5]{ZH2016}).

\item[(4)] $\Gamma$ is derived-finite.

Note that $\Gamma_1$ and $\Gamma_2$ are gentle trees.
By the classification of gentle one-cycle algebra given in \cite[Section 7]{AG2008}
(or by using \cite[Theorem 4.6]{LZ2021}),
we obtain that $\Gamma_1$ and $\Gamma_2$ are derived equivalent to the hereditary Nakayama algebra of type $\mathbb{A}$.
It follows that $\Gamma_1$ and $\Gamma_2$ are derived-finite, and so is $A$.
Therefore, $\der \mathbf{D}^{b}(\modcat \Gamma) = 0.$
\end{itemize}

Now, by Theorem \ref{arrow removal}, we obtain
\[ 0 = \IT (\Gamma) \leq \IT(A)\leq 2\IT(\Gamma)+1 = 1;  \]
\[ 0 = \dim (\Gamma) \leq \dim(A) \leq 2\dim(\Gamma)+1 = 1 \]
and
\[ 0 = \der \mathbf{D}^{b}(\modg) \leq \der\mathbf{D}^{b}(\modcat A)\leq 2\der \mathbf{D}^{b}(\modg)+1 \le 1. \]
Note that $\dim(A) \geq 1$ since $A$ is representation-infinite; hence we have $\dim(A) = 1$.
Moreover, since $A$ is derived-infinite, we have $\der \mathbf{D}^{b}(\modg) > 0$. Then $\der\mathbf{D}^{b}(\modcat A)=1$.

\end{example}

Before presenting the second example, we require the following lemma.
\begin{lemma}\label{lem:dim>0}
	Let $A$ be a finite dimensional algebra over an algebraically closed field $k.$
	If the global dimension $\mathrm{gl.dim}A$ of $A$ is infinite, then $\der \mathbf{D}^{b}(\modA)\geq 1.$
\end{lemma}

\begin{proof}
	Assume that $\der \mathbf{D}^{b}(\modA)=0$.
	Then by \cite{CYZ} that $A$ is an iterated tilted algebra of Dykin type.
	It follows from \cite[Theorem 2.10]{Happel} that there exists an triangle equivalence $\mathbf{D}^{b}(\modA)\simeq\mathbf{D}^{b}({\rm mod}\text{-}kQ)$ for some Dykin quiver $Q.$
	Since $\mathrm{gl.dim}~kQ<\infty$, we have $\mathrm{gl.dim}A<\infty$ by \cite{Happel1} or \cite[Theorem 3.2]{Xi}. This leads to a contradiction.
	Hence $\der \mathbf{D}^{b}(\modA)\geq 1.$
\end{proof}

\begin{example}
This example is taken from \cite[Example 6.3]{GP1}.
Let $\Lambda$ be the $k$-algebra given by the following quiver
	
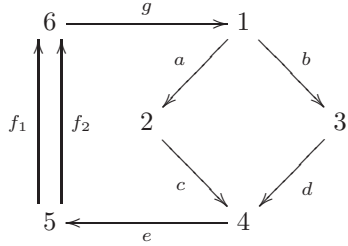
\[\xymatrix{
		6\ar[rr]^g && 1\ar[dl]_a\ar[dr]^b & \\
		& 2\ar[dr]_c & & 3\ar[dl]^d\\
		5\ar@<1ex>[uu]^{f_1}\ar@<-1ex>[uu]_{f_2} & & 4\ar[ll]^e & }\]
	with relations $\{ ac - bd, cef_1, de, ef_1g, f_1gb, ga\}$.
	Then $\Lambda$ is a representation-infinite algebra and $\mathrm{gl.dim}\Lambda=\infty$ by \cite[Example 6.3]{GP1}
or \cite[Example 4.8]{QS}.
	Since $f_{2}$ is not occurring in any relations, $\Lambda$ can be reduced by arrow removal.
	Note that $\Lambda/\langle f_{2}\rangle$ is of finite representation type (see \cite[Example 6.3]{GP1}). By
Definition \ref{def:IT distance} and Definition \ref{extension-dimension}, we can see that $\IT (\Lambda/\langle
f_{2}\rangle)=0=\dim (\Lambda/\langle f_{2}\rangle)$. And we get $\der \mathbf{D}^{b}({\rm mod}\text{-}\Lambda/\langle
f_{2}\rangle)\leq 1$ by \cite[Theorem]{han2009derived}.
	By \cite[Corollary 5.8]{Be2}, we have $\mathrm{gl.dim}\Lambda/\langle f_{2}\rangle=\infty$. It follows from Lemma
\ref{lem:dim>0} that $\der \mathbf{D}^{b}({\rm mod}\text{-}\Lambda/\langle f_{2}\rangle)=1.$

	$(1)$ By Theorem \ref{arrow removal}(1), we know that $\IT (\Lambda)\leq 2\IT (\Lambda/\langle f_{2}\rangle)+1=1$,
that is, $\Lambda$ is an Igusa-Todorov algebra.
	
	$(2)$
	By Theorem \ref{arrow removal}(2), we have $\dim (\Lambda)\leq 2\dim(\Lambda/\langle f_{2}\rangle)+1=2\times 0+1=1$.
Since $\Lambda$ is a representation infinite algebra, we know from \cite{Be2} that $\dim (\Lambda)\geq1$. Thus $\dim
(\Lambda)=1.$
	
	$(3)$ By Theorem \ref{LL-and-gldim}, we get
	$\der\mathbf{D}^{b}(\modL) \leq \inf\{{\rm LL}(\Lambda)-1,{\rm gl.dim} \Lambda\}=\inf\{6-1,+\infty\}=5.$ 	It follows from	 Theorem \ref{arrow removal}(3) that $1\leq\der\mathbf{D}^{b}(\modL)\leq \der\mathbf{D}^{b}({\rm mod}\text{-}\Lambda/\langle f_{2}\rangle)\times 2+1=3 $. 	That is to say, compared to Theorem \ref{LL-and-gldim}, we can sometimes obtain a better upper bound. \end{example}

\section*{Acknowledgments}
\noindent
We thank the anonymous referee for their careful reading and numerous pertinent comments, which helped improve the exposition.

\bigskip


\noindent\textbf{Yajun Ma} \\
$^{1}$School of Mathematics and Physics, Lanzhou Jiaotong University, Lanzhou, 730070, Gansu Province, P. R. China. \\
$^{2}$Gansu Center for Fundamental Research in Complex Systems Analysis and Control, Lanzhou Jiaotong University, Lanzhou, 730070, Gansu Province, P. R. China. \\
E-mail: \textsf{yjma@mail.lzjtu.cn}


\noindent\textbf{Junling Zheng} \\
Department of Mathematics, China Jiliang University, Hangzhou, 310018, Zhejiang Province, P. R. China. \\
E-mail: \textsf{zhengjunling@cjlu.edu.cn}


\noindent\textbf{Yu-Zhe Liu} \\
School of Mathematics and statistics, Guizhou University, 550025 Guiyang, Guizhou, P. R. China. \\
E-mail: \textsf{liuyz@gzu.edu.cn / yzliu3@163.com}

\end{document}